\documentclass[11pt]{article}
	\usepackage{caption}
    \usepackage{url}
    \usepackage{verbatim}
	\usepackage{amsmath}
	\usepackage{amsthm}
	\usepackage{amsfonts}
	\usepackage{graphicx}
	\usepackage{float}
	\usepackage{subfig}
	\usepackage{nicefrac}
	\usepackage{color}
	\usepackage[usenames,dvipsnames,svgnames,table]{xcolor}
	\usepackage{soul}
	\usepackage{calc}
	\usepackage{mathtools}
	\usepackage{mathrsfs}   
	\usepackage{xparse}

\newlength\dlf

	\def\captionfont{\setb@se{11pt}\protect\footnotesize}
    \def\captionfont{\protect\footnotesize}

    \newcommand{\iprd}[2]{\left( #1 , #2 \right)}

	\newcommand{\aIPh}[2]{a_h^{IP}\left( #1 , #2 \right)}

	\newcommand{\aiprd}[2]{a\!\left( #1 , #2 \right)}

    \newcommand{\TTh}{\mathcal{T}_h}
	\newcommand{\Eh}{\mathscr{E}_h}

	\def\norm#1#2{\left\| #1 \right\|_{#2}}

	\newcommand{\ddt}{\delta_{\Delta t}}
	\newcommand{\phih}{\phi_h}

	\newcommand{\ds}{\displaystyle}

	\newtheorem{thm}{Theorem}[section]

	\newtheorem{lem}[thm]{Lemma}
	\newtheorem{rem}[thm]{Remark}

\DeclareSymbolFont{bbold}{U}{bbold}{m}{n}
\DeclareSymbolFontAlphabet{\mathbbold}{bbold}
\newcommand{\err}{\mathbbold{e}}
\newcommand{\errP}{\mathbbold{e}_P}

\newcommand{\errh}{\mathbbold{e}_h}
\NewDocumentCommand{\dgal}{sO{}m}{%
  \IfBooleanTF{#1}
    {\dgalext{#3}}
    {\dgalx[#2]{#3}}%
}
\NewDocumentCommand{\dgalext}{m}{%
  \sbox0{%
    \mathsurround=0pt % just for safety
    $\left\{\vphantom{#1}\right.\kern-\nulldelimiterspace$%
  }%
  \sbox2{\{}%
  \ifdim\ht0=\ht2
    \{\kern-.625\wd2 \{#1\}\kern-.625\wd2 \}%
  \else
    \left\{\kern-.5\wd0\left\{#1\right\}\kern-.5\wd0\right\}%
  \fi
}

\NewDocumentCommand{\dgalx}{om}{%
  \sbox0{\mathsurround=0pt$#1\{$}%
  \sbox2{\{}%
  \ifdim\ht0=\ht2
    \{\kern-.625\wd2 \{#2\}\kern-.625\wd2 \}%
  \else
    \mathopen{#1\{\kern-.5\wd0 #1\{}
    #2
    \mathclose{#1\}\kern-.5\wd0 #1\}}
  \fi
}
\usepackage[only,llbracket,rrbracket,llparenthesis,rrparenthesis]{stmaryrd}

\begin{document}
\title{Continuous data assimilation and long-time accuracy in a $C^0$ interior penalty method for the Cahn-Hilliard equation}

	\author{
Amanda E. Diegel\thanks{Department of Mathematics and Statistics, Mississippi State University, Mississippi State, MS 39562 (adiegel@math.msstate.edu)}
\and
Leo G. Rebholz\thanks{Department of Mathematical Sciences, Clemson University, Clemson, SC 29634 (rebholz@clemson.edu); partially supported by NSF grant DMS 2011490} }

	\maketitle
	
	\numberwithin{equation}{section}
	
\begin{abstract}
We propose a numerical approximation method for the Cahn-Hilliard equations that incorporates continuous data assimilation in order to achieve long time accuracy.  The method
uses a C$^0$ interior penalty spatial discretization of the fourth order Cahn-Hilliard equations, together with a backward Euler temporal discretization.  We prove the method is long time stable and long time accurate, for arbitrarily inaccurate initial conditions, provided enough data measurements are incorporated into the simulation.  Numerical experiments illustrate the effectiveness of the method on a benchmark test problem.
\end{abstract}

\medskip

\section{Introduction}
This paper considers a numerical scheme that incorporates continuous data assimilation (CDA) in order to obtain long time stable and accurate approximations to the Cahn-Hilliard (CH) equation, which is given by 
\cite{Cahn:61:spinodal,CH:58:free,Miranville:19:CH}
\begin{subequations}\label{eq:continuous-problem}
\begin{align}
\partial_t \varphi - \Delta \left(\varphi^3 - \varphi\right) + \varepsilon^2  \Delta^2 \varphi & = 0, \\
\varphi(0) &=\varphi_0,
\end{align}
\end{subequations}
where $\varphi$ represents the order parameter which takes on values between $-1$ and $1$ and is often interpreted as a concentration of one component in a two component system. The states $\varphi = \pm 1$ indicate phases of pure concentration and $\varepsilon > 0$ can be interpreted as an interfacial width between the two phases.  The CH equation arises in many applications across science and engineering, including the modeling of two phase fluid flow, Hele-Shaw flows, copolymer fluids, and crystal growth as just a few examples (cf.  \cite{CS:2018:CHPhaseField,CLSWW:2018:CHWillmore,CMW:2011:MinCahnHilliard,Feng:2006:NSCH,LLG:2002:HeleShaw,vTBVL:2009:PFCM} and the references therein.) Solving \eqref{eq:continuous-problem} analytically in dimensions 2 or 3 is generally very challenging and practitioners typically obtain solutions with numerical simulation methods.  Most numerical methods are developed to acquire approximations to solutions of a mixed weak formulation of problem \eqref{eq:continuous-problem}, see for example
 \cite{Miranville:19:CH,GT:15:survey} and references therein. Additionally, a few papers have developed numerical methods for the fourth-order formulation shown above, \cite{AKW:15:adaptiveDGCH,ADSV:16:CHvirtual,Chen:19:CHdirect}. The method proposed herein is novel in that it is the first to incorporate CDA and, as such, is the first to admit provable long-time accuracy.

In certain problem settings, partial observable or measurement values of the solution may be available.  In such circumstances, using data assimilation to incorporate known solution values into numerical simulations often allows for more stable or accurate solutions.  This has been studied for many different evolutionary physical partial differential equation (PDE) systems in recent years \cite{Daley_1993_atmospheric_book,Law_Stuart_Zygalakis_2015_book,Azouani2014ContinuousDA,LRZ19,Kalman_1960_JBE}.  A new type of data assimilation known as CDA was developed in 2014 \cite{Azouani2014ContinuousDA}, which adapted classical nudging methods of the 1970's (see, e.g., \cite{Anthes_1974_JAS,Hoke_Anthes_1976_MWR}) to use a spatial interpolation operator in the feedback control.  This seemingly small change has led to a profound impact in accuracy and theory of data assimilation methods, as CDA provides mathematically rigorous justification of data assimilated solutions converging to true solutions exponentially fast in time (for arbitrarily inaccurate initial conditions), as well as long time accuracy and stability; these properties are unique among existing data assimilation techniques.  CDA has so far been used to improve solutions in Navier-Stokes equations \cite{Lunasin_Titi_2015,Biswas_Martinez_2017,Farhat_Lunasin_Titi_2016abridged,DMB20,DMB18},  with noisy data  \cite{Bessaih_Olson_Titi_2015,Foias_Mondaini_Titi_2016}, and with temporal and spatial discretizations \cite{LRZ19,RZ19,Ibdah_Mondaini_Titi_2018uniform,GNT18}, for NS-$\alpha$ and Leray-$\alpha$ models \cite{Albanez_Nussenzveig_Lopes_Titi_2016,FLT19}, for Benard convection \cite{Altaf_Titi_Knio_Zhao_Mc_Cabe_Hoteit_2015,Farhat_Jolly_Titi_2015,Farhat_Lunasin_Titi_2016benard}, for
the Brinkman Forchheimer-extended Darcy \cite{Markowich_Titi_Trabelsi_2016} equation, for the surface quasi-geostrophic equation in  \cite{Jolly_Martinez_Titi_2017}, and for weather prediction \cite{DHKTLD19}, among others.  Convergence of discretizations of CDA models has been studied in \cite{LRZ19,RZ19,ZRSI19,Ibdah_Mondaini_Titi_2018uniform,GNT18} for fluid related models, and it was found that if there is enough measurement data, then computed solutions will converge to the true solution exponentially fast in time, up to (optimal) discretization error.

We propose and analyze herein a particular discretization for the Cahn-Hilliard (CH) equation together with CDA, in an effort to obtain long time accuracy and stability of computed solutions to CH equations.  To our knowledge, there is no literature for CDA applied to the CH equation.  Perhaps the main reason for this is that the typical CDA application and analysis does not seem possible (at least not to the authors of this paper) for the second order mixed CH equation, which is a more commonly used formulation than the fourth order formulation above.  At the PDE level, the CDA system we consider takes the form 
\begin{subequations}\label{eq:continuous-da-problem}
\begin{align}
\partial_t \phi - \Delta \left(\phi^3 - \phi\right) + \varepsilon^2  \Delta^2 \phi + \omega I_H(\phi - \varphi) & = 0, \\
\phi(0)&=\phi_0,
\end{align}
\end{subequations}
where $\phi$ is the approximate concentration and $\varepsilon$ is the same as above. The initial condition can be arbitrarily inaccurate, and a common choice is $\phi_0=0$ in cases when there is no a priori knowledge of the initial state.  The scalar $\omega>0$ is known as the nudging parameter, and $I_H$ is the interpolation operator, where $H$ is the resolution of the coarse spatial mesh which represents the locations where measurements are taken (so that $I_H(\varphi)$ is known). The added data assimilation term forces (or nudges) the coarse spatial scales of the approximating solution $\phi$ toward the coarse spatial scales of the true solution $\varphi$.  

The discretization method we choose is a C$^0$ interior penalty method for the spatial discretization, and first order semi-implicit in time.  The first order temporal discretization is chosen for simplicity, and extension to BDF2 can be done in the usual way, following e.g. \cite{LRZ19}.  This spatial discretization is chosen so that C$^0$ finite elements (FEs) can be used, since they  are widely available in FE software but C$^1$ FEs are not, and we note that extension of this work to a C$^1$ FE discretization is possible and in fact the analysis would be simpler.  We prove that for sufficiently small $H$, solutions to our proposed discretization of \eqref{eq:continuous-da-problem} will converge (up to discretization error) to the solution of \eqref{eq:continuous-problem}, exponentially fast in time and for any initial condition in $L^2(\Omega)$.  This in turn provides for long time accuracy and long time stability.

The remainder of the paper proceeds as follows. In Section \ref{sec:preliminaries}, we introduce the necessary notation and preliminary results needed in the proceeding sections. Section \ref{sec:fully-discrete-fem} introduces a fully discrete finite element method (FEM) for the data assimilation model above. We then prove that the FEM is uniquely solvable and demonstrate the long time stability of the scheme. In Section \ref{sec:convergence}, we include a convergence analysis and we conclude with a few numerical experiments supporting our analyses in Section \ref{sec:numerical-experiments}.

\section{Notation and Preliminaries}\label{sec:preliminaries}

We consider a bounded open domain $\Omega \subset \mathbb{R}^2$.  While the method can be used in 3D, the analysis with C$^0$ interior penalty methods is currently limited to 2D. The $L^2$ inner product is denoted $\iprd{\cdot}{\cdot}$. Additionally, we denote the natural function space for the concentration $\varphi$ by $H_N^2(\Omega) := \{ v \in H^2(\Omega) | \partial_n v =0 \}$, where $\partial_n v$ represents the outward unit normal derivative of $v$. Furthermore, we denote a bilinear form $\aiprd{\cdot}{\cdot}: H_N^2(\Omega) \times H_N^2(\Omega) \rightarrow \mathbb{R}$, which is defined for all $v, w \in H_N^2(\Omega)$ by 
\begin{align}\label{eq:aiprd-def}
\aiprd{w}{v} := \int_\Omega \nabla^2 w : \nabla^2 v \, dx,
\end{align}
and which represents the inner product of the Hessian matrices of $w$ and $v$. With this notation, we are able to define a weak formulation of \eqref{eq:continuous-problem} as follows: Find $\varphi \in L^{\infty}(0,T;H_N^2(\Omega))$ such that, for almost all $t \in (0,\infty)$,
\begin{align}\label{eq:weak-formulation}
\iprd{\partial_t \varphi}{\psi} + \iprd{\nabla \left(\varphi^3 - \varphi\right)}{\nabla \psi} + \varepsilon^2  \aiprd{ \varphi}{\psi} = 0, \quad \forall \,  \psi \in H^2_N(\Omega).
\end{align}

\subsection{Discretization Preliminaries}\label{sec:discretization-preliminaries}

Let $\TTh$ be a simplicial triangulation of $\Omega$. We will use the following notation throughout the paper:
\begin{itemize}\itemsep0ex
\item $h_K = $ diameter of triangle $K$ ($h = \max_{K \in \TTh} h_K$),
\item $v_K = $ restriction of the function $v$ to the triangle $K$,
\item $|K| = $ area of the triangle $K$,
\item $\Eh = $ the set of the edges of the triangles in $\TTh$,
\item $e = $ the edge of a triangle.
\end{itemize}
Additionally, we let $Z_h := \{ v \in C(\overline{\Omega}) | v_K = v|_K \in P_2(\Omega) \forall K \in \TTh \}$ represent a standard Lagrange FE space associated with $\TTh$ and assume that the mesh is sufficiently regular for the inverse inequality to hold. Specifically, we assume that there exists a constant $C$ such that for all $v \in Z_h$, 
\begin{align*}
\norm{\nabla v}{L^2} \le C h^{-1} \norm{v}{L^2}.
\end{align*}
Furthermore, we consider $I_H$ to be an interpolation operator that satisfies: For a given mesh $\mathcal{T}_H$ with $H \le 1$ and associated FE space $Z_H$,
\begin{align}
\norm{I_H(w) - w}{L^2} &\le C_I H \norm{\nabla w}{L^2}
\label{eq:projection-properties-diff}
\\
\norm{I_H(w)}{L^2} &\le C_I \norm{w}{L^2},
\label{eq:projection-properties-equiv}
\end{align}
for any $w \in H^1(\Omega)$.  Examples of such $I_H$ are the $L^2$ projection onto $Z_H$ where $Z_H$ consists of piecewise constants over  $\mathcal{T}_H$, the algebraic nudging technique from \cite{RZ19}, and the Scott-Zhang interpolant \cite{SZ90}. 

Let $w \in H^2(K), v \in H^1(K)$, then we have the following integration by parts formula:
\begin{align}
\int_{K} (\Delta w) v \, dx =&\, \int_{\partial K} \dfrac{\partial w}{\partial n} v \, ds - \int_{K} \nabla w \cdot \nabla v \, dx,
\label{eq:integration-by-parts-lower-order}
\end{align}
If instead, $w \in H^4(K), v \in H^2(K)$, then we have:
\begin{align}
\int_K (\Delta^2 w) v \, dx =& \, \int_{\partial K} \left[ \left(\frac{\partial \Delta w}{\partial n}\right) v  - \left(\frac{\partial^2 w}{\partial n^2}\right) \left(\frac{\partial v}{\partial n}\right) - \left(\frac{\partial^2 w}{\partial n \partial t}\right) \left(\frac{\partial v}{\partial t}\right) \right] dS
\nonumber
\\
&+ \int_K \left(\nabla^2 w : \nabla^2 v\right) dx,
\label{eq:integration-by-parts-higher-order}
\end{align}
where $ \nicefrac{\partial}{\partial n}$ (resp. $\nicefrac{\partial}{\partial t}$) denote the exterior normal derivative (resp. the counterclockwise tangential derivative). The integration by parts formula \eqref{eq:integration-by-parts-higher-order} leads to the definition of the bilinear form $\aIPh{\cdot}{\cdot}$ on the piecewise Sobolev space $H^3(\Omega, \TTh) := \{v \in L^2(\Omega) | v_K \in H^3(K) \forall K \in \TTh\}$ such that
\begin{align}
\label{eq:aIPh-def}
\aIPh{w}{v} :=& \, \sum_{K \in \TTh}  \int_K \left(\nabla^2 w : \nabla^2 v\right) \, dx + \sum_{e \in \Eh} \int_e \dgal[\Bigg]{\frac{\partial^2 w}{\partial n_e^2} } \left\llbracket \frac{\partial v}{\partial n_e} \right\rrbracket dS 
\nonumber
\\
&+ \sum_{e \in \Eh} \int_e \dgal[\Bigg]{\frac{\partial^2 v}{\partial n_e^2} } \left\llbracket \frac{\partial w}{\partial n_e} \right\rrbracket dS + \sigma \sum_{e \in \Eh} \frac{1}{|e|} \int_e  \left\llbracket \frac{\partial w}{\partial n_e} \right\rrbracket  \left\llbracket \frac{\partial v}{\partial n_e} \right\rrbracket dS,
\end{align}
with $\sigma \ge 1$ known as a penalty parameter. The jumps and averages that appear in \eqref{eq:aIPh-def} are defined as follows. For an interior edge $e$ shared by two triangles $K_\pm$ where $n_e$ points from $K_-$ to $K_+$, we define on the edge $e$
\begin{align}
\label{eq:jumps-ave-interior}
\left\llbracket \frac{\partial v}{\partial n_e} \right\rrbracket = n_e \cdot \left(\nabla v_+ - \nabla v_-\right),
\quad
\dgal[\Bigg]{\frac{\partial^2 v}{\partial n_e^2} } = \frac{1}{2}\left( \frac{\partial^2 v_-}{\partial n_e^2} +  \frac{\partial^2 v_+}{\partial n_e^2} \right),
\end{align}
where $\displaystyle \frac{\partial^2 v}{\partial n_e^2} = n_e \cdot \left(\nabla^2 v\right) n_e$ and $v_\pm = v |_{K_\pm}$. For a boundary edge $e$, we take $n_e$ to be the unit normal pointing towards the outside of $\Omega$ and define on the edge $e$
\begin{align}\label{eq:jumps-ave-boundary}
\left\llbracket \frac{\partial v}{\partial n_e} \right\rrbracket = - n_e \cdot \nabla v_K ,
\quad
\dgal[\Bigg]{\frac{\partial^2 v}{\partial n_e^2} } = n_e \cdot \left( \nabla^2 v \right) n_e.
\end{align}

\begin{rem}
Note that the definitions \eqref{eq:jumps-ave-interior} and \eqref{eq:jumps-ave-boundary} are independent of the choice of $K_\pm$, or equivalently, independent of the choice of $n_e$.
\end{rem}

Let $\norm{\cdot}{2,h}$ be defined by
\begin{align}\label{def:coip-norm}
\norm{v}{2,h}^2:=  \sum_{K\in \TTh} |v|_{H^2(K)}^2 + \sum_{e\in \Eh} \frac{\sigma}{|e|} \norm{\left\llbracket \frac{\partial v}{\partial n_e} \right\rrbracket}{L^2(e)}^2.
\end{align}
The following theorem guarantees the boundedness of $\aIPh{\cdot}{\cdot}$.

\begin{lem}[Boundedness of $\aIPh{\cdot}{\cdot}$]\label{lem:aIPh-boundedness} There exists positive constants $C_{cont}$ and $C_{coer}$ such that for choices of the penalty parameter $\sigma$ large enough we have
\begin{align}
\aIPh{w}{v} &\le C_{cont} \norm{w}{2,h} \norm{v}{2,h} \quad \forall \, w, v \, \in Z_h  \label{eq:cont}
\\
C_{coer}\norm{w}{2,h}^2 &\le \aIPh{w}{w} \quad \forall \, w \in Z_h, \label{eq:coer}
\end{align}
where the constants $C_{cont}$ and $C_{coer}$ depend only on the shape regularity of $\TTh$.
\end{lem}
\begin{proof}
The proof of the Lemma may be found in \cite{brenner:11:frontiers}.
\end{proof}

We remark that for all $v \in Z_h$, we have the following Poincar\'{e} type inequalities: There exists a constant $C_{P}$ depending only on $\Omega$ such that,
\begin{align*}
\norm{v}{L^2} \le C_P \norm{\nabla v}{L^2} \quad \text{and} \quad \norm{\nabla v}{L^2} \le C_P \norm{v}{2,h}.
\end{align*}
Additionally, the first of these inequalities holds for all $v \in H_N^2(\Omega)$. Finally, the following two lemmas are critical to the remainder of the paper.

\begin{lem}\label{lem:grad-split}
Suppose $\Omega$ is a bounded polygonal domain. For all $w \in Z_h$, $v \in H^1(\Omega)$, and $\sigma$ large enough, 
\begin{align}
    |\iprd{\nabla w}{\nabla v}| \le \sqrt{2} \norm{w}{2,h} \norm{v}{L^2}.
\end{align}
\end{lem}

\begin{proof}
We begin by rewriting the integration by part formula \eqref{eq:integration-by-parts-lower-order}:
\begin{align*}
\int_{K} \nabla w \cdot \nabla v \, dx = \int_{\partial K} \dfrac{\partial w}{\partial n} v \, ds - \int_{K} \Delta w v \, dx.
\end{align*}
Summing over all triangles in $\TTh$, we have
\begin{align*}
\sum_{K \in \TTh} \int_{K} \nabla w \cdot \nabla v \, dx = \sum_{K \in \TTh} \int_{\partial K} \dfrac{\partial w}{\partial n} v \, ds - \sum_{K \in \TTh} \int_{K} \Delta w v \, dx.
\end{align*}
Now, we can write the first sum on the right-hand side of the equation above as a sum over the edges in $\Eh$:
\begin{align*}
\sum_{K \in \TTh} \int_{K} \nabla w \cdot \nabla v \, dx = \sum_{e \in \Eh} \int_{e} \left\llbracket \frac{\partial w}{\partial n_e} \right\rrbracket v \, ds - \sum_{K \in \TTh} \int_{K} \Delta w v \, dx.
\end{align*}
Using the Cauchy-Schwarz inequality and a standard trace inequality, we have
\begin{align*}
\left(\sum_{K \in \TTh} \int_{K} \nabla w \cdot \nabla v \, dx\right)^2 \le&\,  2 \left(\sum_{e \in \Eh} |e|^{-1} \norm{\left\llbracket \frac{\partial w}{\partial n_e} \right\rrbracket}{L^2(e)}^2\right) \left(\sum_{e \in \Eh} |e| \norm{v}{L^2(e)}^2 \right)
\\
&+2 \left(\sum_{K \in \TTh} |w|_{H^2(K)}^2 \right)\left(\sum_{K \in \TTh} \norm{v}{L^2(K)}^2 \, dx\right)
\\
\le&\, 2 \left(\sum_{e \in \Eh} |e|^{-1} \norm{\left\llbracket \frac{\partial w}{\partial n_e} \right\rrbracket}{L^2(e)}^2 \right) \left(C \sum_{K \in \TTh} \norm{v}{L^2(K)} \right)
\\
&+ 2\left(\sum_{K \in \TTh}|w|_{H^2(K)}^2 \right)\left(\sum_{K \in \TTh} \norm{v}{L^2(K)} \right)  
\\
&\le\, 2 \norm{w}{2,h}^2 \norm{v}{L^2(\Omega)}^2
\end{align*}
for $\sigma$ large enough.
\end{proof}

\begin{lem}\label{lem:geometric-series-bound}
Suppose the constants $r$ and $B$ satisfy $r > 1$ and $B \ge 0$. Then if the sequence of real numbers $\{a_m\}$ satisfies 
\begin{align}
r a_{m+1} \le a_m + B,
\end{align}
we have that
\begin{align*}
a_{m+1} \le a_0 \left(\frac{1}{r}\right)^{m+1} + \frac{B}{r-1}.
\end{align*}
\end{lem}

\begin{proof}
See \cite{LRZ19}.
\end{proof}

%\begin{lem}{\bf The Discrete Gronwall Inequality} \cite{HR:90:NonstatNS,Layton:08:IntroVisFlow}:	
%	\label{lem:discrete-gronwall}
%Fix $T>0$, and suppose $\left\{a^m\right\}_{m=1}^M$, $\left\{b^m\right\}_{m=1}^M$ and $\left\{c^m\right\}_{m=1}^{M-1}$ are non-negative sequences such that $\tau\sum_{m=1}^{M-1} c^m \le C_1$, where $C_1$ is independent of $\tau$ and $M$, and $M\tau = T$.  Suppose that, for all $\tau>0$, 
%	\begin{equation}
%a^M + \tau \sum_{m=1}^{M} b^m \le C_2 +\tau \sum_{m=1}^{M-1} a^m c^m ,
%	\label{eq:gronwall-assumption}
%	\end{equation}
%where $C_2>0$ is a constant independent of $\tau$ and $M$.  Then, for all $\tau>0$, 
%	\begin{equation}
%a^M +\tau \sum_{m=1}^{M} b^m \le C_2 \exp \left(\tau\sum_{m=1}^{M-1}c^m \right) \le C_2\exp(C_1).
%	\label{eq:gronwall-conclusion}
%	\end{equation}
%Note that the sum on the right-hand-side of \eqref{eq:gronwall-assumption} must be explicit.
%	\end{lem}
%

\section{Fully Discrete C$^0$ Interior Penalty FEM with Data Assimilation }\label{sec:fully-discrete-fem}

Let $M$ be a positive integer and $0=t_0 < t_1 < \cdots < t_M = T$ be a uniform partition of $[0,T]$. A fully discrete C$^0$ interior penalty method for \eqref{eq:continuous-da-problem} is: Given $\phih^{m-1} \in Z_h$ and true solution $\varphi \in L^\infty(0,T;H_N^2(\Omega))$, find $\phih^{m} \in  Z_h$ such that 
\begin{align}\label{eq:fully-discrete-fem}
\iprd{\ddt \phih^{m}}{ \psi} + \iprd{\nabla \left(\left(\phih^m\right)^3 - \phih^{m-1}\right)}{\nabla \psi} + \varepsilon^2 \aIPh{\phih^{m}}{\psi} + \omega \iprd{I_H \left(\phih^m - \varphi^m\right)}{\psi} =0 ,
\end{align}
for all $\psi \in Z_h$ with initial data taken to be $\phi_h^0 := P_h \phi_0$ where $P_h: H_N^2(\Omega) \rightarrow Z_h$ is a Ritz projection operator such that
\begin{align}\label{eq:H2-Ritz-projection}
\aIPh{P_h \phi - \phi}{\xi} = 0 \quad \forall \, \xi \in Z_h, \quad \iprd{P_h \phi - \phi}{1} = 0,
\end{align}
and where $\displaystyle \ddt \phih^{m} := \frac{\phih^{m} - \phih^{m-1}}{\Delta t}$ with $\Delta t=T/M$ as the size for the time step.  We will refer to the method \eqref{eq:fully-discrete-fem} as the CDA-FEM.

We will begin by showing that solutions to \eqref{eq:fully-discrete-fem} exist followed by a stability result and then conclude with a proof for the uniqueness of the solution. 

\begin{lem}
Let $\phih^{m-1} \in Z_h$ and $\varphi^m \in H_N^2(\Omega)$ be given. Then, if $\omega < \dfrac{C_{coer}\varepsilon^2}{2 C_P^2 C_I^2 H^2}$, there exists a solution $\phih^m \in Z_h$ to \eqref{eq:fully-discrete-fem}.
\end{lem}

\begin{proof}
Let $\phih^{m-1} \in Z_h$ and $\varphi^m \in H_N^2(\Omega)$ be given and $G_h: Z_h \rightarrow Z_h$ be the continuous map defined by
\begin{align}\label{eq:fixed-point-form}
    \iprd{G_h(\phih^m)}{\chi} :=&\, \iprd{\phih^{m} - \phih^{m-1}}{ \chi} + \Delta t \iprd{\nabla \left(\left(\phih^m\right)^3 - \phih^{m-1}\right)}{\nabla \chi} 
    \nonumber
    \\
    &+ \varepsilon^2 \Delta t \, \aIPh{\phih^{m}}{\chi} + \omega \Delta t \iprd{I_H \left(\phih^m - \varphi^m\right)}{\chi}.
\end{align}
It is a well-known consequence of Brouwer's fixed-point theorem \cite{Thomee:84:galerkin} that \newline $G_h(\phih^m) = 0$ has a solution $\phih^m \in B_q := \{ \chi \in Z_h : \, \norm{\chi}{L^2} \le q\}$ if $\iprd{G_h(\chi)}{\chi} > 0$ for $\norm{\chi}{L^2} = q$, where we define $q^2 =  \left(1 + \dfrac{4 \Delta t}{C_{coer} \varepsilon^2} \right) \norm{\phih^{m-1}}{L^2}^2 + \omega C_I^2 \Delta t \norm{ \varphi^m}{L^2}^2$. 

Using Young's inequality, \eqref{eq:projection-properties-diff}, \eqref{eq:projection-properties-equiv}, and Lemmas \ref{lem:aIPh-boundedness} and \ref{lem:grad-split}, we have
\begin{align*}
\iprd{G_h(\chi)}{\chi} \ge& \, (1+ \omega \Delta t) \norm{\chi}{L^2}^2 + 3 \Delta t \norm{\chi \nabla \chi}{L^2}^2 + C_{coer} \varepsilon^2 \Delta t \norm{\chi}{2,h}^2 
\\
&- \iprd{\phih^{m-1}}{\chi} - \Delta t \iprd{\nabla \phih^{m-1}}{\nabla \chi} - \omega \Delta t \iprd{\chi - I_H \chi}{\chi} - \omega \Delta t \iprd{I_H \varphi^m}{\chi}
\\
\ge &\, (1+ \omega \Delta t) \norm{\chi}{L^2}^2 + 3 \Delta t \norm{\chi \nabla \chi}{L^2}^2 + C_{coer} \varepsilon^2 \Delta t \norm{\chi}{2,h}^2 - \norm{\phih^{m-1}}{L^2} \norm{\chi}{L^2} 
 \\
&- \sqrt{2} \Delta t \norm{\phih^{m-1}}{L^2} \norm{\chi}{2,h} - \omega \Delta t \norm{\chi - I_h \chi}{L^2} \norm{\chi}{L^2} - \omega \Delta t \norm{I_H \varphi^m}{L^2} \norm{\chi}{L^2}
\\
\ge &\, (1+ \omega \Delta t) \norm{\chi}{L^2}^2 + 3 \Delta t \norm{\chi \nabla \chi}{L^2}^2 + C_{coer} \varepsilon^2 \Delta t \norm{\chi}{2,h}^2 - \norm{\phih^{m-1}}{L^2} \norm{\chi}{L^2} 
 \\
&- \sqrt{2} \Delta t \norm{\phih^{m-1}}{L^2} \norm{\chi}{2,h} - \omega C_I H \Delta t \norm{\nabla \chi}{L^2} \norm{\chi}{L^2} - \omega C_I \Delta t \norm{ \varphi^m}{L^2} \norm{\chi}{L^2}
\\
\ge &\, (1+ \omega \Delta t) \norm{\chi}{L^2}^2 + 3 \Delta t \norm{\chi \nabla \chi}{L^2}^2 + C_{coer} \varepsilon^2 \Delta t \norm{\chi}{2,h}^2 - \norm{\phih^{m-1}}{L^2} \norm{\chi}{L^2} 
 \\
&- \sqrt{2} \Delta t \norm{\phih^{m-1}}{L^2} \norm{\chi}{2,h} - \omega C_P C_I H \Delta t \norm{\chi}{2,h} \norm{\chi}{L^2} - \omega C_I \Delta t \norm{ \varphi^m}{L^2} \norm{\chi}{L^2}
\\
\ge &\, (1+ \omega \Delta t) \norm{\chi}{L^2}^2 + 3 \Delta t \norm{\chi \nabla \chi}{L^2}^2 + C_{coer} \varepsilon^2 \Delta t \norm{\chi}{2,h}^2 - \frac{1}{2}\norm{\phih^{m-1}}{L^2}^2 - \frac{1}{2} \norm{\chi}{L^2}^2
 \\
&- \frac{2 \Delta t}{C_{coer} \varepsilon^2} \norm{\phih^{m-1}}{L^2}^2 - \frac{C_{coer} \varepsilon^2 \Delta t}{4}\norm{\chi}{2,h}^2 -  \frac{C_{coer} \varepsilon^2 \Delta t}{4}\norm{\chi}{2,h}^2 
\\
&- \frac{\omega^2 C_P^2 C_I^2 H^2 \Delta t}{C_{coer} \varepsilon^2} \norm{\chi}{L^2}^2 - \frac{\omega C_I^2 \Delta t}{2} \norm{ \varphi^m}{L^2}^2 - \frac{\omega \Delta t}{2} \norm{\chi}{L^2}^2.
\end{align*}
Combining like terms and multiplying by 2 now provides the bound
\begin{align*}
2\iprd{G_h(\chi)}{\chi} \ge& \, \left[1 + \omega \Delta t \left(1 - \frac{2 \omega C_P^2 C_I^2 H^2 }{C_{coer} \varepsilon^2} \right) \right] \norm{\chi}{L^2}^2 + 6 \Delta t \norm{\chi \nabla \chi}{L^2}^2 + C_{coer} \varepsilon^2 \Delta t \norm{\chi}{2,h}^2 
\\
&- \norm{\phih^{m-1}}{L^2}^2 - \frac{4 \Delta t}{C_{coer} \varepsilon^2} \norm{\phih^{m-1}}{L^2}^2 - \omega C_I^2 \Delta t \norm{ \varphi^m}{L^2}^2 
\\
\ge&\, 0,
\end{align*}
if $0 < \omega < \dfrac{C_{coer} \varepsilon^2}{2 C_P^2 C_I^2 H^2 }$ and $\norm{\chi}{L^2}^2 = q^2 = \left(1 + \dfrac{4 \Delta t}{C_{coer} \varepsilon^2} \right) \norm{\phih^{m-1}}{L^2}^2 + \omega C_I^2 \Delta t \norm{ \varphi^m}{L^2}^2 $.
\end{proof}

\begin{lem}\label{lem:stability}
Let $\varphi \in L^\infty(0,\infty;H_N^2(\Omega))$ represent the true solution and black $H$ and $\omega$ be chosen so that
\begin{align}
\lambda_0 := \frac{\omega C_{coer} \varepsilon^2 - 2 \omega^2 C_I^2 C_P^2 H^2 - 4}{C_{coer} \varepsilon^2 + 4 \Delta t} > 0, \label{wH0}
\end{align} 
i.e. $H$ sufficiently small and $\omega$ sufficiently large.
%\begin{align}
%&H < \frac{C_{coer}\varepsilon^2}{4 \sqrt{2} C_I C_P} ,
%\label{eq:cond-on-H-2}
%\\
%&0 < \frac{C_{coer}\varepsilon^2 - \alpha}{4 C_I^2 C_P^2 H^2} < \omega < \frac{C_{coer} \varepsilon^2 + \alpha}{4 C_P^2 C_I^2 H^2 } ,
%\label{eq:cond-on-omega-2}
%\end{align}
%where $\alpha^2 := C_{coer}^2\varepsilon^4 - 32 C_I^2 C_P^2 H^2$. 
Then, for any $m,h, \Delta t >0$, solutions to the CDA-FEM \eqref{eq:fully-discrete-fem} satisfy
\begin{align*}
\norm{\phih^m}{L^2}^2 \le \norm{\phih^0}{L^2}^2 \left(\frac{1}{1 + \lambda \Delta t}\right)^m + \frac{ \omega \,  \varepsilon^2 C_I^2}{\omega \varepsilon^2 - 2 \omega^2 C_I^2 C_P^2 H^2 -4} \, \Phi \le C_{data},
\end{align*}
where $\Phi := \norm{\varphi}{L^{\infty}(0,\infty;L^2(\Omega))}^2$.
\end{lem}

\begin{proof}
Setting $\psi = \phih^{m}$ in \eqref{eq:fully-discrete-fem} we have
\begin{align*}
\iprd{\ddt \phih^{m}}{ \phih^m} + 3 \norm{\phih^{m} \nabla \phih^m}{L^2}^2 - \iprd{\nabla \phih^{m-1}}{\nabla \phih^m} + \varepsilon^2 \aIPh{\phih^{m}}{\phih^m} + \omega \iprd{I_H (\phih^m - \varphi^m)}{\phih^m}  =0.
\end{align*}
Now, adding and subtracting appropriate terms and using Lemma \ref{lem:grad-split} along with properties \eqref{eq:projection-properties-diff}, \eqref{eq:projection-properties-equiv}, and \eqref{eq:coer},  Young's inequality, and the polarization identity,
%the fact that $a(a-b) = \frac{1}{2} \left(a^2 - b^2 + (a - b)^2\right)$, 
we have
\begin{align*}
&\frac{1}{2\Delta t} \left(\norm{\phih^m}{L^2}^2 - \norm{\phih^{m-1}}{L^2}^2 + \norm{\phih^m - \phih^{m-1}}{L^2}^2\right) + \omega \norm{\phih^m}{L^2}^2 + 3 \norm{\phih^{m} \nabla \phih^m}{L^2}^2 + C_{coer} \varepsilon^2 \norm{\phih^{m}}{2,h}^2 
\\
&\quad\le \omega \iprd{\phih^m - I_H \phih^m}{ \phih^m} + \omega \iprd{I_H \varphi^m}{ \phih^m} + \iprd{\nabla \phih^{m-1}}{\nabla \phih^m}
\\
&\quad\le \omega \norm{\phih^m - I_H \phih^m}{L^2} \norm{\phih^m}{L^2} + \omega \norm{I_H \varphi^m}{L^2} \norm{\phih^m}{L^2} + \sqrt{2}\norm{\phih^{m-1}}{L^2} \norm{\phih^m}{2,h}
\\
&\quad\le \omega \, C_I H \norm{\nabla \phih^m}{L^2} \norm{ \phih^m}{L^2} + \omega \, C_I \norm{\varphi^m}{L^2} \norm{ \phih^m}{L^2} +  \sqrt{2} \norm{\phih^{m-1}}{L^2} \norm{\phih^m}{2,h}
\\
&\quad\le \omega \, C_I C_P H \norm{\phih^m}{2,h} \norm{ \phih^m}{L^2} + \omega \, C_I \norm{\varphi^m}{L^2} \norm{ \phih^m}{L^2} + \sqrt{2}\norm{\phih^{m-1}}{L^2} \norm{\phih^m}{2,h}
\\
&\quad\le \frac{\omega^2 \, C_I^2 C_P^2 H^2}{C_{coer} \varepsilon^2} \norm{ \phih^m}{L^2}^2 + \frac{C_{coer} \varepsilon^2}{2} \norm{\phih^m}{2,h}^2 + \frac{\omega C_I^2}{2}\norm{\varphi^m}{L^2}^2 + \frac{\omega}{2} \norm{ \phih^m}{L^2}^2 + \frac{2}{C_{coer}\varepsilon^2} \norm{ \phih^{m-1}}{L^2}^2 .
\end{align*}
Combining like terms, multiplying by $2\Delta t$, and dropping %a few of the 
positive terms on the left hand side of the equation above, we get that
\begin{align*}
\left[1 + \omega \, \Delta t - \frac{2 \omega^2 \, C_I^2 C_P^2 H^2 \, \Delta t}{C_{coer} \varepsilon^2}\right] \norm{\phih^m}{L^2}^2 
\le  \frac{C_{coer} \varepsilon^2 + 4 \Delta t}{C_{coer} \varepsilon^2} \norm{\phih^{m-1}}{L^2}^2  + \omega \, C_I^2 \Delta t \norm{\varphi^m}{L^2}^2.
\end{align*}
Multiplying by $\frac{C_{coer} \varepsilon^2}{C_{coer} \varepsilon^2 + 4 \Delta t} $, we obtain
\begin{align*}
\left[\frac{C_{coer} \varepsilon^2 (1+ \omega \Delta t) - 2 \omega^2 C_I^2 C_P^2 H^2 \Delta t }{C_{coer} \varepsilon^2 + 4 \Delta t}\right] \norm{\phih^m}{L^2}^2 &\le  \norm{\phih^{m-1}}{L^2}^2  +  \frac{\omega \, C_I^2 C_{coer} \varepsilon^2 \Delta t}{C_{coer} \varepsilon^2 + 4 \Delta t} \norm{\varphi^m}{L^2}^2,
\end{align*}
which leads to 
\begin{align*}
\left[1+ \Delta t \left( \frac{\omega C_{coer} \varepsilon^2 - 2 \omega^2 C_I^2 C_P^2 H^2 - 4}{C_{coer} \varepsilon^2 + 4 \Delta t} \right) \right] \norm{\phih^m}{L^2}^2&\le  \norm{\phih^{m-1}}{L^2}^2  + \frac{\omega \, C_I^2 C_{coer} \varepsilon^2 \Delta t}{C_{coer} \varepsilon^2 + 4 \Delta t}  \Phi.
\end{align*}
 Requiring $\ds \lambda_0 := \frac{\omega C_{coer} \varepsilon^2 - 2 \omega^2 C_I^2 C_P^2 H^2 - 4}{C_{coer} \varepsilon^2 + 4 \Delta t}> 0$ and applying Lemma \ref{lem:geometric-series-bound} yields the desired result: for any $m, h, \Delta t >0$
\begin{align*}
\norm{\phih^m}{L^2}^2 \le \norm{\phih^0}{L^2}^2 \left(\frac{1}{1 + \lambda_0 \Delta t}\right)^m + \frac{ \omega C_I^2 C_{coer} \varepsilon^2}{\omega C_{coer}\varepsilon^2 - 2 \omega^2 C_I^2 C_P^2 H^2 -4 } \, \Phi \le C_{data}.
\end{align*}
\end{proof}

\begin{rem}
Lemma \ref{lem:stability} is satisfied if the nudging parameter $\omega$ is $\mathcal{O}(\varepsilon^{-2})$ which would then require $H$ to be $\mathcal{O}(\varepsilon^2)$. However, the numerical experiments provided in Section \ref{sec:numerical-experiments} suggests that $H$ can be taken much larger than that.
\end{rem}

Define $C_{inf}:=\sup_{0\le m < \infty} \| \phi_h^m \|_{L^{\infty}}$.  Then under the conditions from Lemma \ref{lem:stability}, we have $C_{inf}<\infty$ on any regular mesh since
by the inverse inequality and the 2D Agmon inequality, 
\[
C_{inf} = \sup_{0\le m < \infty} \| \phi_h^m \|_{L^{\infty}} \le \sup_{0\le m < \infty} Ch^{-1}  \| \phi_h^m \|_{L^{2}} \le CC_{data}h^{-1}.
\]
While this is the best long-time $L^{\infty}$ bound we were able to prove, we expect $C_{inf}=O(1)$ and not $O(h^{-1})$, since computing CH in practice yields
$|\phi_h^m(x)|\le 1.1$ or $1.2$, and never as high as 2.  Furthermore, if we assume a finite end time $T$, then it is likely that we can prove this with the usual techniques \cite{DFW:15:CHDS}
that $ \max_{0\le m\le T/\Delta t} \| \phi_h^m \|_{L^{\infty}}\le C e^T$, with $C$ depending on data and independent of $h$ and $\Delta t$.  Since CH solutions generally converge quickly to a steady state solution, we expect 
$ \max_{0\le m\le T/\Delta t} \| \phi_h^m \|_{L^{\infty}}= \sup_{0\le m < \infty} \| \phi_h^m \|_{L^{\infty}},$ from which we again infer that $C_{inf}$ will be independent of $h$.

%\begin{asmp}\label{asmp:Linf-boundedness}
%We assume that for any $h, \Delta t >0$, solutions to the CDA-FEM \eqref{eq:fully-discrete-fem} satisfy
%\begin{align*}
%\norm{\phih^m}{L^{\infty}} \le C'_{data},
%\end{align*}
%where $C'_{data}$ is independent of the time and spatial step sizes $\tau$ and $h$. This assumption is critical to the proof for uniqueness and to the error analysis.
%\end{asmp}

\begin{lem}\label{lem:uniqueness}
Let $\varphi \in L^\infty(0,{\infty};H_N^2(\Omega))$ represent the true solution and suppose that %$H$ and $\omega$ satisfy
%\begin{align}
%&H <  \frac{C_{coer} \varepsilon^2}{6\sqrt{2} C_P C_I (C'_{data})^2},
%\label{eq:cond-on-H-2}
%\\
%&0 \le \frac{C_{coer} \varepsilon^2 - \beta_1}{2 C_I^2C_P^2H^2} < \omega < \frac{C_{coer} \varepsilon^2 + \beta_1}{2 C_I^2C_P^2H^2} ,
%\label{eq:cond-on-omega-2}
%\end{align}
%where $ \beta_1^2 = C_{coer}^2 \varepsilon^4 - 144 C_P^{10} C_I^2 (C'_{data})^4 H^2$. 
\begin{equation}
\frac{1}{\Delta t}  + \omega >  \left( \frac{C_I^2 C_P^2 H^2 \omega^2 + 18 (C_{inf})^4}{C_{coer} \varepsilon^2}  \right). \label{cfl1}
\end{equation}
Then solutions to the CDA-FEM \eqref{eq:fully-discrete-fem} are unique.
\end{lem}

\begin{rem}
The condition \eqref{cfl1} is satisfied with a sufficiently small $\Delta t$, or by $\omega$ being sufficiently large while $H$ is sufficiently small.  We note that 
this is a sufficient condition.
\end{rem}

\begin{proof}
Let $\phih^{m-1} \in Z_h$ and $\varphi^m \in H_N^2(\Omega)$ be given. Let $U$ and $V$ represent two solutions to \eqref{eq:fully-discrete-fem}. Then,
\begin{align*}
\frac{1}{\Delta t} \iprd{U-V}{ \psi} + \iprd{\nabla (U^3 - V^3)}{\nabla \psi} + \varepsilon^2 \aIPh{U-V}{\psi} + \omega \iprd{I_H \left(U - V\right)}{\psi} =0 .
\end{align*}
Setting $\psi = U-V$ and invoking Lemma \ref{lem:grad-split}, we have 
\begin{align*}
&\left(\frac{1}{\Delta t} + \omega \right) \norm{U-V}{L^2}^2 +  C_{coer} \varepsilon^2 \norm{U-V}{2,h}^2  
\\
&\quad= - \iprd{\nabla (U^3-V^3)}{\nabla (U-V)} + \omega \iprd{(U-V) - I_H (U- V)}{U-V} 
\\
&\quad\le \sqrt{2}\norm{U^3-V^3}{L^2} \norm{U-V}{2,h} + \omega \norm{(U-V) - I_H (U- V)}{L^2}  \norm{U-V}{L^2}
\\
&\quad\le  \sqrt{2} \norm{ U^2 + UV + V^2}{L^\infty} \norm{U-V}{L^2} \norm{U-V}{2,h} + C_I H \omega \norm{\nabla (U- V)}{L^2}  \norm{U-V}{L^2}
\\
&\quad\le 3\sqrt{2} (C_{inf})^2 \norm{U-V}{L^2} \norm{U-V}{2,h} + C_I C_P H \omega \norm{U- V}{2,h}  \norm{U-V}{L^2}
\\
&\quad\le \frac{18 (C_{inf})^4}{C_{coer} \varepsilon^2} \norm{U-V}{L^2}^2 + \frac{C_{coer}\varepsilon^2}{2} \norm{U-V}{2,h}^2 + \frac{C_I^2 C_P^2 H^2 \omega^2}{C_{coer} \varepsilon^2}  \norm{U-V}{L^2}^2
\end{align*}
Combining like terms leads to
\begin{align*}
\left( \left(\frac{1}{\Delta t}  + \omega \right) -  \left( \frac{C_I^2 C_P^2 H^2 \omega^2 + 18 (C_{inf})^4}{C_{coer} \varepsilon^2}  \right) \right) \norm{U-V}{L^2}^2 + \frac{C_{coer} \varepsilon^2}{2} \norm{U-V}{2,h}^2 \le 0.
\end{align*}
Therefore, the solution to \eqref{eq:fully-discrete-fem} is unique under the condition that \eqref{cfl1} holds.
\end{proof}

\section{Error Estimates}\label{sec:convergence}

We are now in a position to prove that the global in time error estimates may be established in the $L^2$ norm. We provide a rigorous convergence analysis for the semi-discrete method in the appropriate energy norms.  Note that the CDA-FEM \eqref{eq:fully-discrete-fem} is not well-defined for solutions to \eqref{eq:continuous-da-problem} since $Z_h \not\subset H_N^2(\Omega)$. Therefore, we define $W_h \subset H_N^2(\Omega)$ to be the Hsieh-Clough-Tocher micro finite element space associated with $\TTh$ as in \cite{brenner:11:frontiers}. We furthermore define the linear map $E_h: Z_h \rightarrow W_h \cap H_N^2(\Omega)$ as in \cite{brenner:11:frontiers} which allows us to consider the following problem: Find $\varphi \in H_N^2(\Omega) $ such that
\begin{align}
\label{eq:cont-correction}
&\iprd{\partial_t \varphi}{\psi} + \iprd{\nabla \left(\varphi^3 - \varphi\right)}{\nabla \psi} + \varepsilon^2 \aIPh{\varphi}{\psi} 
\nonumber
\\
&\, =  \iprd{\partial_t \varphi}{\psi - E_h \psi} + \iprd{\nabla \left(\varphi^3 - \varphi\right)}{\nabla (\psi - E_h \psi)} + \varepsilon^2 \aIPh{\varphi}{ \psi - E_h \psi}, \quad \forall \, \psi \in Z_h .
\end{align}
Solutions of \eqref{eq:cont-correction} are consistent with solutions of \eqref{eq:weak-formulation} since $\aIPh{\varphi}{E_h \psi} = \aiprd{\varphi}{E_h \psi}$ for all $\psi \in Z_h$. 

We introduce the following notation:
\begin{align*}
\err^{m} = \errP^{m} + \errh^{m}, \quad \errP^{m} := \varphi^m - P_h \varphi^m, \quad \errh^{m} := P_h \varphi^m - \phih^m,
\end{align*}
where $\varphi^m := \varphi(t_m)$. Using this notation and subtracting \eqref{eq:fully-discrete-fem} from \eqref{eq:cont-correction}, we have for all $\psi \in Z_h$
\begin{align*}
&\iprd{\ddt \err^{m}}{\psi} + \varepsilon^2 \aIPh{\err^{m}}{\psi} + \omega \iprd{I_H \err^{m}}{\psi} +\iprd{\nabla \left((\varphi^m)^3 - (\phih^m)^3 \right)}{\nabla \psi} 
\\
&\quad= \iprd{\ddt \varphi^{m} - \partial_t \varphi^{m}}{\psi} + \iprd{\nabla \err^{m-1}}{\nabla \psi} +\iprd{\nabla (\varphi^m - \varphi^{m-1})}{\nabla \psi} 
\\
&\qquad+ \iprd{\partial_t \varphi^{m}}{\psi - E_h \psi} + \iprd{\nabla \left(\left(\varphi^m\right)^3 - \varphi^{m}\right)}{\nabla (\psi - E_h \psi)} + \varepsilon^2 \aIPh{\varphi^{m}}{ \psi - E_h \psi} .
\end{align*}
Invoking the properties of the Ritz projection operator, we have for all $\psi \in Z_h$
\begin{align}
&\iprd{\ddt \errh^{m}}{\psi} + \varepsilon^2 \aIPh{\errh^{m}}{\psi} + \omega \iprd{I_H \errh^{m}}{\psi} +\iprd{\nabla \left( (\varphi^m)^3 - (\phih^m)^3 \right)}{\nabla \psi} = -\iprd{\ddt \errP^{m}}{\psi} 
\nonumber
\\
&\quad - \omega \iprd{I_H \errP^{m}}{\psi} + \iprd{\ddt \varphi^{m} - \partial_t \varphi^{m}}{\psi} + \iprd{\nabla (\errP^{m-1} + \errh^{m-1})}{\nabla \psi} + \iprd{\nabla (\varphi^m - \varphi^{m-1})}{\nabla \psi} 
\nonumber
\\
&\qquad+ \iprd{\partial_t \varphi^{m}}{\psi - E_h \psi} + \iprd{\nabla \left(\left(\varphi^m\right)^3 - \varphi^{m}\right)}{\nabla (\psi - E_h \psi)} + \varepsilon^2 \aIPh{\varphi^{m}}{ \psi - E_h \psi} .
\label{eq:error-formulation-semi}
\end{align}
Adding and subtracting appropriate terms and setting $\psi = \errh^{m}$, we arrive at the key error equation
\begin{align}\label{eq:main-error}
&\iprd{\ddt \errh^{m}}{\errh^{m}} + \varepsilon^2 \aIPh{\errh^{m}}{\errh^{m}} + \omega \iprd{\errh^{m}}{\errh^{m}} 
\nonumber
\\
&\quad= -\iprd{\ddt \errP^{m}}{\errh^{m}}  - \omega \iprd{I_H \errP^{m}}{\errh^{m}} + \omega \iprd{\errh^{m} - I_H \errh^{m}}{\errh^{m}} + \iprd{\ddt \varphi^{m} - \partial_t \varphi^{m}}{\errh^{m}} 
\nonumber
\\
&\qquad- \iprd{ \nabla \left( (\varphi^m)^3 - (\phih^m)^3 \right)}{\nabla \errh^m} + \iprd{\nabla (\errP^{m-1} + \errh^{m-1})}{\nabla \errh^{m}} + \iprd{\nabla (\varphi^m - \varphi^{m-1})}{\nabla \errh^{m}} 
\nonumber
\\
&\qquad \quad+ \iprd{\partial_t \varphi^{m}}{\errh^{m} - E_h \errh^{m}} + \iprd{\nabla \left(\left(\varphi^m\right)^3 - \varphi^{m}\right)}{\nabla (\errh^{m} - E_h \errh^{m})} + \varepsilon^2 \aIPh{\varphi^{m}}{ \errh^{m} - E_h \errh^{m}} .
\end{align}

The following lemma will bound many of the terms on the right hand side of \eqref{eq:main-error} by oscillations in the time derivative of the concentration $\partial_t \varphi$. The procedure, known as a medius analysis, has been utilized in much of the literature found on the C$^0$-IP method and details can be found in \cite{brenner:11:frontiers}. It relies on an equivalent formulation of the bilinear form $\aIPh{\cdot}{\cdot}$ for functions satisfying $w \in H^4(\Omega,\TTh) \cap H^1(\Omega)$ and $v \in H^2(\Omega, \TTh) \cap H^1(\Omega)$:
\begin{align}
\label{eq:aIPh-def-alt}
\aIPh{w}{v} :=& \, \sum_{K \in \TTh}  \int_K \left(\Delta^2 w\right) v \, dx + \sum_{e \in \Eh} \int_e \dgal[\Bigg]{\frac{\partial^2 v}{\partial n_e^2} } \left\llbracket \frac{\partial w}{\partial n_e} \right\rrbracket dS 
\nonumber
\\
&+ \sum_{e \in \Eh} \int_e \left( \left\llbracket \frac{\partial \Delta w}{\partial n_e} \right\rrbracket v -  \left\llbracket \frac{\partial^2 w}{\partial n_e^2} \right\rrbracket \dgal[\Bigg]{\frac{\partial v}{\partial n_e} } - \left\llbracket \frac{\partial^2 w}{\partial n_e \partial t_e} \right\rrbracket \frac{\partial v}{\partial t_e} \right) dS
\nonumber
\\
&+ \sigma \sum_{e \in \Eh} \frac{1}{|e|} \int_e  \left\llbracket \frac{\partial w}{\partial n_e} \right\rrbracket  \left\llbracket \frac{\partial v}{\partial n_e} \right\rrbracket dS,
\end{align}
where $H^s(\Omega,\TTh) := \{v \in L^2(\Omega) | v_K \in H^s(K) \forall K \in \TTh\}$.

\begin{lem}\label{lem:bounds-oscillations}
Suppose $\varphi^{m}$ is a weak solution to \eqref{eq:weak-formulation}. Then for any $m, h,\tau > 0$,
\begin{align*}
& \iprd{\partial_t \varphi^{m}}{\errh^{m} - E_h \errh^{m}} + \iprd{\nabla \left(\left(\varphi^m\right)^3 - \varphi^{m}\right)}{\nabla (\errh^{m} - E_h \errh^{m})} + \varepsilon^2 \aIPh{\varphi^{m}}{ \errh^{m} - E_h \errh^{m}}
\\
&\quad\le\frac{C}{C_{coer} \varepsilon^2} \left( \left[\text{Osc}_j(\partial_t \varphi^m)\right]^2 + \norm{ \varphi^{m} - P_h \varphi^m}{2,h}^2 \right) + \frac{C_{coer} \varepsilon^2}{8} \norm{\errh^{m}}{2,h}^2
\end{align*}
for $t^* \in (t_{m-1}, t_{m})$ where $\text{Osc}_j(\nu)$ is referred to as the oscillation of $\nu$ (of order $j$) defined by 
\begin{align} \label{eq:data-osc}
\text{Osc}_j(\nu):= \left( \sum\limits_{K \in \TTh} h^4 \norm{\nu - \tilde{\nu}}{L^2(K)}^2 \right)^{\frac{1}{2}}
 \end{align}
and where $ \tilde{\nu}$ is the $L^2$ orthogonal projection of $\nu$ on $P_j(\Omega, \TTh)$, the space of piecewise polynomial functions of degree less than or equal to $j$, i.e.,
\begin{align*}
\int_\Omega (\nu -  \tilde{\nu}) \psi \, dx = 0 \quad \forall \, \psi \in P_j(\Omega, \TTh).
\end{align*}
\end{lem}

\begin{proof}
Properties of the Ritz projection operator \eqref{eq:H2-Ritz-projection} lead to,
\begin{align}\label{eq:aiph-split}
&\aIPh{\varphi^m}{\errh^{m} - E_h \errh^{m}} 
\nonumber
\\
&\quad= \aIPh{\varphi^m - P_h \varphi^m}{\errh^{m} -E_h \errh^{m}} + \aIPh{P_h \varphi^m}{\errh^{m} -E_h \errh^{m}} 
\nonumber
\\
&\quad= -\aIPh{\varphi^m - P_h \varphi^m}{E_h \errh^{m}} + \aIPh{P_h \varphi^m}{\errh^{m} -E_h \errh^{m}} 
\nonumber
\\
&\quad = \, \sum_{K \in \TTh}  \int_K \nabla^2 (P_h \varphi^m - \varphi^m) : \nabla^2 (E_h \errh^m) dx +  \sum_{e \in \Eh} \int_e \dgal[\Bigg]{\frac{\partial^2  \left(E_h \errh^{m}\right)}{\partial n_e^2} } \left\llbracket \frac{\partial P_h \varphi^{m}}{\partial n_e} \right\rrbracket dS
\nonumber
\\
&\qquad + \aIPh{P_h \varphi^m}{\errh^{m} -E_h \errh^{m}} 
\end{align}
Furthermore, the alternative definition \eqref{eq:aIPh-def-alt} yields the following
\begin{align}\label{eq:aiph-with-projection}
& \aIPh{P_h \varphi^{m}}{ \errh^{m} - E_h \errh^{m}} = \, \sum_{K \in \TTh}  \int_K \left(\Delta^2 P_h \varphi^m \right) \left(\errh^{m} -E_h \errh^{m}\right) \, dx 
\nonumber
\\
&\quad+ \sum_{e \in \Eh} \int_e \dgal[\Bigg]{\frac{\partial^2  \left(\errh^{m} -E_h \errh^{m}\right)}{\partial n_e^2} } \left\llbracket \frac{\partial P_h \varphi^{m}}{\partial n_e} \right\rrbracket dS + \sum_{e \in \Eh} \int_e \left\llbracket \frac{\partial \Delta P_h \varphi^{m}}{\partial n_e} \right\rrbracket \left(\errh^{m} -E_h \errh^{m}\right) dS
\nonumber
\\
&\quad - \sum_{e \in \Eh} \int_e  \left\llbracket \frac{\partial^2 P_h \varphi^m}{\partial n_e^2} \right\rrbracket \dgal[\Bigg]{\frac{\partial  \left(\errh^{m} -E_h \errh^{m}\right)}{\partial n_e} } dS - \sum_{e \in \Eh} \int_e  \left\llbracket \frac{\partial^2 P_h \varphi^m}{\partial n_e \partial t_e} \right\rrbracket \frac{\partial  \left(\errh^{m} -E_h \errh^{m}\right)}{\partial t_e} dS
\nonumber
\\
&\quad+ \sigma \sum_{e \in \Eh} \frac{1}{|e|} \int_e  \left\llbracket \frac{\partial P_h \varphi^m}{\partial n_e} \right\rrbracket  \left\llbracket \frac{\partial \left(\errh^{m} -E_h \errh^{m}\right)}{\partial n_e} \right\rrbracket dS.
\end{align}
Combining equations \eqref{eq:aiph-split}--\eqref{eq:aiph-with-projection}, we have
\begin{align*}
&\iprd{\partial_t \varphi^{m}}{\errh^{m} - E_h \errh^{m}} + \iprd{\nabla \left(\left(\varphi^m\right)^3 - \varphi^{m}\right)}{\nabla (\errh^{m} - E_h \errh^{m})} + \varepsilon^2 \aIPh{\varphi^{m}}{ \errh^{m} - E_h \errh^{m}}
\\
&\quad =  \sum_{K \in \TTh}  \int_K \left(\varepsilon^2 \Delta^2 P_h \varphi^m - \Delta \left( \left(\varphi^m\right)^3 - \varphi^{m} \right) + \partial_t \varphi^{m}\right) \left(\errh^{m} -E_h \errh^{m}\right) \, dx\
\\
&\qquad + \varepsilon^2 \sum_{K \in \TTh}  \int_K \nabla^2 (P_h \varphi^m - \varphi^m) : \nabla^2 (E_h \errh^m) dx + \varepsilon^2 \sum_{e \in \Eh} \int_e \dgal[\Bigg]{\frac{\partial^2 \errh^{m}}{\partial n_e^2} } \left\llbracket \frac{\partial P_h \varphi^{m}}{\partial n_e} \right\rrbracket dS 
\\
&\qquad+ \varepsilon^2 \sum_{e \in \Eh} \int_e \left\llbracket \frac{\partial \Delta P_h \varphi^{m}}{\partial n_e} \right\rrbracket \left(\errh^{m} -E_h \errh^{m}\right) dS - \varepsilon^2 \sum_{e \in \Eh} \int_e  \left\llbracket \frac{\partial^2 P_h \varphi^m}{\partial n_e^2} \right\rrbracket \dgal[\Bigg]{\frac{\partial  \left(\errh^{m} -E_h \errh^{m}\right)}{\partial n_e} } dS 
\\
&\qquad - \varepsilon^2 \sum_{e \in \Eh} \int_e  \left\llbracket \frac{\partial^2 P_h \varphi^m}{\partial n_e \partial t_e} \right\rrbracket \frac{\partial  \left(\errh^{m} -E_h \errh^{m}\right)}{\partial t_e} dS + \varepsilon^2 \sigma \sum_{e \in \Eh} \frac{1}{|e|} \int_e  \left\llbracket \frac{\partial P_h \varphi^m}{\partial n_e} \right\rrbracket  \left\llbracket \frac{\partial \left(\errh^{m} -E_h \errh^{m}\right)}{\partial n_e} \right\rrbracket dS.
\end{align*}

Following the medius analysis presented in \cite{brenner:11:frontiers} (see pages 96-100), we proceed by bounding each of the terms on the right-hand side:
\begin{align*}
&\left|\sum_{K \in \TTh}  \int_K \left(\varepsilon^2\Delta^2 P_h \varphi^m - \Delta \left( \left(\varphi^m\right)^3 - \varphi^{m} \right) + \partial_t \varphi^{m}\right) \left(\errh^{m} -E_h \errh^{m}\right) \, dx \right| 
\\
&\qquad\le \left( \sum_{K \in \TTh} h^4 \norm{\varepsilon^2\Delta^2 P_h \varphi^m - \Delta \left( \left(\varphi^m\right)^3 - \varphi^{m} \right) + \partial_t \varphi^{m}}{L^2(K)}^2\right)^{\nicefrac{1}{2}} \norm{\errh^{m}}{2,h},
\\
&\left|\varepsilon^2\sum_{K \in \TTh}  \int_K \nabla^2 (P_h \varphi^m - \varphi^m) : \nabla^2 (E_h \errh^m) dx \right| \le C \varepsilon^2 \left(\sum_{K \in \TTh} |\varphi^m - P_h \varphi^m|_{H^2(K)}^2\right)^{\nicefrac{1}{2}} \norm{\errh^m}{2,h}
\\
& \left|\varepsilon^2\sum_{e \in \Eh} \int_e \dgal[\Bigg]{\frac{\partial^2  \errh^{m} }{\partial n_e^2} } \left\llbracket \frac{\partial P_h \varphi^{m}}{\partial n_e} \right\rrbracket \right| 
\\
&\qquad\le C \varepsilon^2 \left( \sum_{e \in \Eh} \frac{1}{|e|} \norm{\left\llbracket \frac{\partial (P_h \varphi^{m} - \varphi^m)}{\partial n_e} \right\rrbracket}{L^2(e)}^2 \right)^{\nicefrac{1}{2}} \norm{\errh^{m}}{2,h},
\\
&\left|\varepsilon^2\sum_{e \in \Eh} \int_e \left\llbracket \frac{\partial \Delta P_h \varphi^{m}}{\partial n_e} \right\rrbracket \left(\errh^{m} -E_h \errh^{m}\right) dS \right| 
\\
&\qquad\le C \varepsilon^2 \left( \sum_{e \in \Eh} |e|^3 \norm{\left\llbracket \frac{\partial \Delta P_h \varphi^{m}}{\partial n_e} \right\rrbracket }{L^2(e)}^2\right)^{\nicefrac{1}{2}} \norm{\errh^{m}}{2,h},
\\
&\left|\varepsilon^2\sum_{e \in \Eh} \int_e  \left\llbracket \frac{\partial^2 P_h \varphi^m}{\partial n_e^2} \right\rrbracket \dgal[\Bigg]{\frac{\partial  \left(\errh^{m} -E_h \errh^{m}\right)}{\partial n_e} } dS \right| 
\\
&\qquad \le C \varepsilon^2 \left( \sum_{e \in \Eh} |e| \norm{\left\llbracket \frac{\partial^2 P_h \varphi^{m}}{\partial n_e^2} \right\rrbracket }{L^2(e)}^2\right)^{\nicefrac{1}{2}} \norm{\errh^{m}}{2,h},
\\
&\left| \varepsilon^2\sum_{e \in \Eh} \int_e \left\llbracket \frac{\partial^2 P_h \varphi^m}{\partial n_e \partial t_e} \right\rrbracket \frac{\partial  \left(\errh^{m} -E_h \errh^{m}\right)}{\partial t_e} dS \right| 
\\
&\qquad\le C \varepsilon^2 \left( \sum_{e \in \Eh} \frac{1}{|e|} \norm{\left\llbracket \frac{\partial (P_h \varphi^{m} - \varphi^m)}{\partial n_e} \right\rrbracket }{L^2(e)}^2\right)^{\nicefrac{1}{2}} \norm{\errh^{m}}{2,h} ,
\\
&\left| \varepsilon^2 \sigma \sum_{e \in \Eh} \frac{1}{|e|} \int_e  \left\llbracket \frac{\partial P_h \varphi^m}{\partial n_e} \right\rrbracket  \left\llbracket \frac{\partial \left(\errh^{m} -E_h \errh^{m}\right)}{\partial n_e} \right\rrbracket dS \right|
\\
&\qquad\le  C \varepsilon^2 \left( \sum_{e \in \Eh} \frac{1}{|e|} \norm{\left\llbracket \frac{\partial (P_h \varphi^{m} - \varphi^m)}{\partial n_e} \right\rrbracket }{L^2(e)}^2\right)^{\nicefrac{1}{2}} \norm{\errh^{m}}{2,h} .
\end{align*}
Thus, we have
\begin{align*}
&\iprd{\partial_t \varphi^{m}}{\errh^{m} - E_h \errh^{m}} + \iprd{\nabla \left(\left(\varphi^m\right)^3 - \varphi^{m}\right)}{\nabla (\errh^{m} - E_h \errh^{m})} + \varepsilon^2 \aIPh{\varphi^{m}}{ \errh^{m} - E_h \errh^{m}}
\\
&\quad\le C \varepsilon^2 \Bigg(\sum_{K \in \TTh} |\varphi^m - P_h \varphi^m|_{H^2(K)}^2+ \frac{1}{\varepsilon^4} \sum_{K \in \TTh} h^4 \norm{\varepsilon^2 \Delta^2 P_h \varphi^m - \Delta \left( \left(\varphi^m\right)^3 - \varphi^{m} \right) + \partial_t \varphi^{m}}{L^2(K)}^2 
\\
&\qquad + \sum_{e \in \Eh} \frac{1}{|e|} \norm{\left\llbracket \frac{\partial (P_h \varphi^{m} - \varphi^m)}{\partial n_e} \right\rrbracket}{L^2(e)}^2 + \sum_{e \in \Eh} |e|^3 \norm{\left\llbracket \frac{\partial \Delta P_h \varphi^{m}}{\partial n_e} \right\rrbracket }{L^2(e)}^2 
\\
&\qquad + \sum_{e \in \Eh} |e| \norm{\left\llbracket \frac{\partial^2 P_h \varphi^{m}}{\partial n_e^2} \right\rrbracket }{L^2(e)}^2 +  \sum_{e \in \Eh} \frac{1}{|e|} \norm{\left\llbracket \frac{\partial (P_h \varphi^{m} - \varphi^m)}{\partial n_e} \right\rrbracket }{L^2(e)}^2 
\\
&\qquad+ \sum_{e \in \Eh} \frac{1}{|e|} \norm{\left\llbracket \frac{\partial (P_h \varphi^{m} - \varphi^m)}{\partial n_e} \right\rrbracket }{L^2(e)}^2\Bigg)^{\nicefrac{1}{2}} \norm{\errh^{m}}{2,h} 
\\
&\quad \le C \varepsilon^2 \Bigg( \frac{1}{\varepsilon^4} \left[\text{Osc}_j(\partial_t \varphi^m)\right]^2 + \sum_{K \in \TTh} |\varphi^m - P_h \varphi^m |_{H^2(K)}^2 + \sum_{e \in \Eh} \frac{1}{|e|} \norm{\left\llbracket \frac{\partial (P_h \varphi^{m} - \varphi^m)}{\partial n_e} \right\rrbracket }{L^2(e)}^2\Bigg)^{\nicefrac{1}{2}} \norm{\errh^{m}}{2,h}
\\
&\quad \le C \varepsilon^2 \left(\frac{1}{\varepsilon^2} \left[\text{Osc}_j(\partial_t \varphi^m)\right] + \norm{ \varphi^{m} - P_h \varphi^m}{2,h} \right) \norm{\errh^{m}}{2,h}
\end{align*}
where we have followed the medius analysis presented in \cite{brenner:11:frontiers} (see pages 101-106) and where 
$\text{Osc}_j(\nu)$ is referred to as the oscillation of $\nu$ (of order $j$) defined by 
\begin{align} \label{eq:data-osc}
\text{Osc}_j(\nu):= \left( \sum\limits_{K \in \TTh} h^4 \norm{\nu - \tilde{\nu}}{L^2(K)}^2 \right)^{\frac{1}{2}}
 \end{align}
and where $ \tilde{\nu}$ is the $L^2$ orthogonal projection of $\nu$ on the space of piecewise polynomial functions of degree less than or equal to $j$, i.e.,
\begin{align*}
\int_\Omega (\nu -  \tilde{\nu}) \psi \, dx = 0 \quad \forall \,  \psi \in P_j(\Omega, \TTh).
\end{align*}
Thus,
\begin{align}
& \iprd{\partial_t \varphi^{m}}{\errh^{m} - E_h \errh^{m}} + \iprd{\nabla \left(\left(\varphi^m\right)^3 - \varphi^{m}\right)}{\nabla (\errh^{m} - E_h \errh^{m})} + \varepsilon^2 \aIPh{ \varphi^{m}}{ \errh^{m} - E_h \errh^{m}}
\nonumber
\\
&\quad\le \frac{C}{C_{coer}} \left( \frac{1}{\varepsilon^2}\left[\text{Osc}_j(\partial_t \varphi^m)\right]^2 + \norm{ \varphi^{m} - P_h \varphi^m}{2,h}^2 \right) + \frac{C_{coer} \varepsilon^2}{8} \norm{\errh^{m}}{2,h}^2.
\end{align}
\end{proof}

\begin{lem}
Let $u_t \in L^{\infty}(0,\infty;H^\ell(\Omega))$ and $u \in L^{\infty}(0,\infty;H^s(\Omega))$, then the Ritz projection operator \eqref{eq:H2-Ritz-projection} satisfies the following bound:
\begin{align}\label{eq:ritz-proj-bound}
\norm{u - P_h u}{2,h} \le C^{**}_{data} h^{\min(s,k+1)-2}
\end{align}
where $k \ge 2$ is the order of the Lagrange finite element space $Z_h$.
\end{lem}

\begin{proof}
According to \cite{brenner:11:frontiers} and considering the model problem 
\begin{align*}
\Delta^2 u &= f \quad \text{in} \quad \Omega
\\
\frac{\partial u}{\partial n} = \frac{\partial \Delta u}{\partial n} &=0 \quad \text{on} \quad \partial \Omega
\end{align*}
it can be shown that 
\begin{align*}
\norm{u - P_h u}{2,h} \le C h^{\min(s,k+1)-2}
\end{align*}
as long as $f \in H^\ell(\Omega)$ and $u \in H^s(\Omega)$ for $s \in (2,\ell + 4]$, where $k \ge 2$ is the order of the Lagrange finite element space $Z_h$ and where $C$ depends on the data. 

In the case of a parabolic model problem 
\begin{align*}
\Delta^2 u &= u_t \quad \text{in} \quad \Omega,
\\
u(0) &=u_0 \quad \text{in} \quad \Omega,
\\
\frac{\partial u}{\partial n} = \frac{\partial \Delta u}{\partial n} &=0 \quad \text{on} \quad \partial \Omega,
\end{align*}
we simply replace $f$ with $u_t$ in the analysis of \cite{brenner:11:frontiers} in order to achieve \eqref{eq:ritz-proj-bound}.
\end{proof}

We are now in position to prove the main theorem in this section. We shall assume that the weak solutions have the additional regularities.  
\begin{align}
\varphi &\in  L^{\infty}(0,\infty;H^3(\Omega)),
\nonumber
\\
\partial_t \varphi &\in L^{\infty}\left(0,\infty;H^3(\Omega)\right) ,
\nonumber
\\
\partial_{tt} \varphi &\in L^\infty\left(0,\infty;L^2(\Omega)\right).
	\label{eq:higher-regularities}
	\end{align}
With these regularities, we set $\ell = 2$, $s = 3$, and $k=2$ in order to obtain
\begin{align}\label{eq:ritz-proj-bound-1}
\norm{\varphi - P_h \varphi}{2,h} \le C^{**}_{data} h.
\end{align}

\begin{thm}\label{thm:main-error-result}
Let $\varphi^m$ represent the solution to \eqref{eq:fully-discrete-fem} and suppose that $\varphi $ satisfies the regularities \eqref{eq:higher-regularities} and that $H$ and $\omega$ are chosen so that
\begin{align}
\lambda_1 := \dfrac{ C_{coer} \varepsilon^2\omega - 4 C_I^2 C_P^2 H^2 \omega^2 - 72 \left((C_{inf})^2 + (C_{data}^{\prime})^2\right)^2 - 16}{C_{coer} \varepsilon^2  + 16 \Delta t } > 0, \label{wH1}
\end{align} 
i.e. $H$ sufficiently small and $\omega$ sufficiently large.
%\begin{align}
%&H \le \frac{C_{coer} \varepsilon^2}{4 C_I C_P \sqrt{16 + 72 \left(({\color{black} C_{inf} })^2 + (C_{data}^{**})^2\right)^2}},
%\label{eq:cond-on-H-2}
%\\
%& 0 < \frac{C_{coer} \varepsilon^2 - \beta}{8 C_I^2 C_P^2 H^2} < \omega < \frac{C_{coer} \varepsilon^2 - \beta}{8 C_I^2 C_P^2 H^2}  ,
%\label{eq:cond-on-omega-2}
%\end{align}
%where $\beta^2 := (C_{coer} \varepsilon^2)^2 - 8 C_I^2 C_P^2 H^2 \left(16 + 36 C_P^4\left[(C_{data}^*)^2 h^{-2} + C_{Inv}^4 C_{data}^2 h^{-6} \right]\right)$. 
Then we have
\begin{align*}
\norm{\varphi^m - \phih^m}{L^2}^2 &\le \norm{\varphi^0 - \phih^0}{L^2}^2 \left(\frac{1}{1+\lambda_1 \Delta t}\right)^m 
\nonumber
\\
&+ \left(\dfrac{ C_{coer} \varepsilon^2 \left( h^2 + h^5 + (\Delta t)^2\right)}{\omega C_{coer} \varepsilon^2 - 4 C_I^2 C_P^2 H^2 \omega^2 - 72 \left((C_{inf})^2 + (C_{data}^{\prime})^2\right)^2 - 16}\right) C_{data}^{*}
\end{align*}
for any $m, h, \Delta t > 0$.
\end{thm}

\begin{rem}
The sufficient condition \eqref{wH1} is a similar sufficient condition to what is found in the long term error bound CDA applied to Navier-Stokes equations \cite{RZ19,LRZ19}, where $H$ must be small enough so that the nudging parameter can be taken large enough to allow a long term error bound to hold.  In our numerical tests, just as in the numerical tests for CDA applied to Navier-Stokes in \cite{RZ19,LRZ19}, the sufficient condition appears far from a necessary condition.
\end{rem}

\begin{proof}
We proceed by bounding the first seven terms on the right hand side of equation \eqref{eq:main-error}. The first bound follows from an application of Young's Inequality, Taylor's theorem, and standard finite element theory.  We have
\begin{align}
-\iprd{\ddt \errP^{m}}{\errh^{m}}  &\le \norm{\ddt \errP^{m}}{L^2} \norm{\errh^{m}}{L^2}
\nonumber
\\
&\le \frac{2}{\omega} \norm{\ddt \errP^{m}}{L^2}^2 + \frac{\omega}{8} \norm{\errh^{m}}{L^2}^2
\nonumber
\\
&\le \frac{2}{\omega} \norm{\varphi_t(t^*) - P_h \varphi_t(t^*)}{L^2}^2 + \frac{\omega}{8} \norm{\errh^{m}}{L^2}^2,
\nonumber
\\
&\le \frac{2C_P^4 (C^{\prime\prime}_{data})^2}{\omega} h^2 + \frac{\omega}{8} \norm{\errh^{m}}{L^2}^2,
\label{eq:error-first}
\end{align}
where $t^* \in [t_{m-1}, t_m]$ and where $C^{\prime\prime}_{data}$ corresponds to the assumption that $\varphi_t \in L^{\infty}\left(0,\infty;H^3(\Omega)\right)$ from \eqref{eq:higher-regularities}. The next two estimates rely on properties of the projection operator $I_H$, \eqref{eq:projection-properties-diff} and \eqref{eq:projection-properties-equiv}. Thus,  with the assumption that $\varphi \in L^\infty(0,\infty;H^3(\Omega))$, we have
\begin{align}
- \omega \iprd{I_H \errP^{m}}{\errh^{m}}  &\le \omega \norm{I_H \errP^{m}}{L^2} \norm{\errh^{m}}{L^2}
\nonumber
\\
&\le 2 \omega \norm{I_H \errP^{m}}{L^2}^2 + \frac{\omega}{8} \norm{\errh^{m}}{L^2}^2
\nonumber
\\
&\le 2 C_I^2 \omega \norm{\errP^{m}}{L^2}^2 + \frac{\omega}{8} \norm{\errh^{m}}{L^2}^2
\nonumber
\\
&\le 2 C_I^2 C_P^4 \omega \norm{\errP^{m}}{2,h}^2 + \frac{\omega}{8} \norm{\errh^{m}}{L^2}^2
\nonumber
\\
&\le 2 C_I^2 C_P^4 (C^{\prime}_{data})^2 \omega h^2 + \frac{\omega}{8} \norm{\errh^{m}}{L^2}^2
\\
\text{and}
\nonumber
\\
\omega \iprd{\errh^{m} - I_H \errh^{m}}{\errh^{m}}  &\le \omega \norm{\errh^{m} - I_H \errh^{m}}{L^2} \norm{\errh^{m}}{L^2}
\nonumber
\\
 &\le C_I H \omega \norm{\nabla \errh^{m} }{L^2} \norm{\errh^{m}}{L^2}
 \nonumber
\\
 &\le \frac{C_{coer} \varepsilon^2}{8} \norm{\errh^{m} }{2,h}^2 + \frac{2 C_I^2 C_P^2 H^2 \omega^2}{C_{coer} \varepsilon^2} \norm{\errh^{m}}{L^2}^2.
\end{align}
For the next term, an application of Taylor's theorem lead to
\begin{align}
\iprd{\ddt \varphi^{m} - \partial_t \varphi^{m}}{\errh^{m}}  &\le \norm{\ddt \varphi^{m} - \partial_t \varphi^{m}}{L^2}  \norm{\errh^{m}}{L^2}
\nonumber
\\
&\le \frac{2}{\omega}\norm{\ddt \varphi^{m} - \partial_t \varphi^{m}}{L^2}^2  + \frac{\omega}{8} \norm{\errh^{m}}{L^2}^2
\nonumber
\\
&\le \frac{2(\Delta t)^2}{\omega}\norm{\varphi_{ss}(t^*)}{L^2}^2  + \frac{\omega}{8} \norm{\errh^{m}}{L^2}^2,
\nonumber
\\
&\le \frac{C^{\prime\prime\prime}_{data} (\Delta t)^2}{\omega}  + \frac{\omega}{8} \norm{\errh^{m}}{L^2}^2,
\end{align}
where $t^* \in [t_{m-1}, t_m]$ and $C^{\prime\prime\prime}_{data}$ corresponds to the assumption that $\varphi_{tt} \in L^{\infty}(0,\infty;L^2(\Omega))$ from \eqref{eq:higher-regularities}. For the nonlinear term,  we use H\"older and Young's inequalities to obtain,
\begin{align}
- \iprd{ \nabla \left( (\varphi^m)^3 - (\phih^m)^3 \right)}{\nabla \errh^m} &\le \sqrt{2} \norm{(\varphi^m)^3 - (\phih^m)^3}{L^2} \norm{\errh^m}{2,h}
\nonumber
\\
&\le \sqrt{2} \norm{(\varphi^m)^2 +\varphi^m \phih^m + (\phih^m)^2}{L^{\infty}} \norm{\errP^m + \errh^m}{L^2} \norm{\errh^m}{2,h}
\nonumber
\\
&\le \frac{3\sqrt{2} \left((C_{inf})^2 + (C_{data}^{\prime})^2\right)}{2} \left( \norm{\errP^m}{L^2} + \norm{\errh^m}{L^2} \right) \norm{\errh^m}{2,h}
\nonumber
\\
&\le \frac{18 \left((C_{inf})^2 + (C_{data}^{\prime})^2\right)^2}{C_{coer}\varepsilon^2} \left( \norm{\errP^m}{L^2} + \norm{\errh^m}{L^2} \right)^2 + \frac{C_{coer} \varepsilon^2}{8} \norm{\errh^m}{2,h}^2
\nonumber
\\
&\le \frac{36 \left((C_{inf})^2 + (C_{data}^{\prime})^2\right)^2}{C_{coer}\varepsilon^2} \norm{\errP^m}{L^2}^2 
\nonumber
\\
&\quad+ \frac{36 \left((C_{inf})^2 + (C_{data}^{\prime})^2\right)^2}{C_{coer}\varepsilon^2} \norm{\errh^m}{L^2}^2 + \frac{C_{coer} \varepsilon^2}{8} \norm{\errh^m}{2,h}^2,
\nonumber
\\
&\le \frac{36 C_P^4 \left((C_{inf})^2 + (C_{data}^{\prime})^2\right)^2}{C_{coer}\varepsilon^2} \norm{\errP^m}{2,h}^2 
\nonumber
\\
&\quad+ \frac{36 \left((C_{inf})^2 + (C_{data}^{\prime})^2\right)^2}{C_{coer}\varepsilon^2} \norm{\errh^m}{L^2}^2 + \frac{C_{coer} \varepsilon^2}{8} \norm{\errh^m}{2,h}^2,
\nonumber
\\
&\le \frac{36 C_P^4 (C^{**}_{data})^2 \left((C_{inf})^2 + (C_{data}^{\prime})^2\right)^2}{C_{coer}\varepsilon^2} h^2 
\nonumber
\\
&\quad+ \frac{36 \left((C_{inf})^2 + (C_{data}^{\prime})^2\right)^2}{C_{coer}\varepsilon^2} \norm{\errh^m}{L^2}^2 + \frac{C_{coer} \varepsilon^2}{8} \norm{\errh^m}{2,h}^2,
\end{align}
where $C_{data}^{\prime}$ corresponds to the assumption that $\varphi \in L^{\infty}(0,\infty;H^3(\Omega))$.  Again, relying on Lemma \ref{lem:grad-split}, we have
\begin{align}
\iprd{\nabla (\errP^{m-1} + \errh^{m-1})}{\nabla \errh^{m}} & \le \sqrt{2} \norm{\errP^{m-1} + \errh^{m-1}}{L^2} \norm{\errh^{m}}{2,h}
\nonumber
\\
&\le \frac{8}{C_{coer} \varepsilon^2}\norm{\errP^{m-1}}{L^2}^2 + \frac{8}{C_{coer} \varepsilon^2}\norm{\errh^{m-1}}{L^2}^2 + \frac{C_{coer} \varepsilon^2}{8} \norm{\errh^{m}}{2,h}^2
\nonumber
\\
&\le \frac{8C_P^4 (C^{\prime\prime\prime}_{data})^2}{C_{coer} \varepsilon^2}h^2 + \frac{8}{C_{coer} \varepsilon^2}\norm{\errh^{m-1}}{L^2}^2 + \frac{C_{coer} \varepsilon^2}{8} \norm{\errh^{m}}{2,h}^2.
\end{align}
Taylor's Theorem leads to the bounds on the next term:
\begin{align}
-\iprd{\nabla (\varphi^m - \varphi^{m-1})}{\nabla \errh^{m}} &= \iprd{\Delta (\varphi^m - \varphi^{m-1})}{\errh^{m}}
\nonumber
\\
&\le \norm{\Delta (\varphi^m - \varphi^{m-1})}{L^2} \norm{\errh^{m}}{L^2}
\nonumber
\\
&\le \frac{2}{\omega} \norm{\Delta (\varphi^m - \varphi^{m-1})}{L^2}^2 + \frac{\omega}{8} \norm{\errh^{m}}{L^2}^2
\nonumber
\\
&\le \frac{ 2(\Delta t)^2}{\omega} \norm{\Delta \varphi_s(t^*)}{L^2}^2 + \frac{\omega}{8} \norm{\errh^{m}}{L^2}^2,
\nonumber
\\
&\le \frac{ C_{data}^{\prime\prime}(\Delta t)^2}{\omega}  + \frac{\omega}{8} \norm{\errh^{m}}{L^2}^2,
\end{align}
where $t^* \in [t_{m-1}, t_m]$ and we have used the assumption that $\varphi_{t} \in L^{\infty}(0,\infty;H^3(\Omega))$ from assumption \eqref{eq:higher-regularities}. Finally, Lemma \ref{lem:bounds-oscillations} allows us to bound the remaining terms by
\begin{align}
& \iprd{\partial_t \varphi^{m}}{\errh^{m} - E_h \errh^{m}} + \iprd{\nabla \left(\left(\varphi^m\right)^3 - \varphi^{m-1}\right)}{\nabla (\errh^{m} - E_h \errh^{m})} + \varepsilon^2 \aIPh{\phi^{m}}{ \errh^{m} - E_h \errh^{m}}
\nonumber
\\
&\hspace{2in}\le \frac{C}{C_{coer}} \left( \frac{1}{\varepsilon^2} \left[\text{Osc}_j(\partial_t \varphi^m)\right]^2 + \norm{ \varphi^{m} - P_h \varphi^m}{2,h}^2 \right) + \frac{C_{coer} \varepsilon^2}{8} \norm{\errh^{m}}{2,h}^2
\nonumber
\\
&\hspace{2in}\le \frac{C}{C_{coer}} \left( \frac{1}{\varepsilon^2} \left[\text{Osc}_j(\partial_t \varphi^m)\right]^2 + (C^{**}_{data})^2h^2 \right) + \frac{C_{coer} \varepsilon^2}{8} \norm{\errh^{m}}{2,h}^2.
\label{eq:error-last}
\end{align}

Combining inequalities \eqref{eq:error-first}--\eqref{eq:error-last} with Lemma \ref{lem:aIPh-boundedness}, leads to
\begin{align}
&\frac{1}{2\Delta t} \left(\norm{\errh^{m}}{L^2}^2 - \norm{\errh^{m-1}}{L^2}^2 + \norm{\errh^{m} - \errh^{m-1}}{L^2}^2\right) + \frac{C_{coer} \varepsilon^2}{2} \norm{\errh^{m}}{2,h}^2 + \frac{\omega}{2} \norm{\errh^{m}}{L^2}^2 
\nonumber
\\
&\quad\le \frac{8}{C_{coer} \varepsilon^2 }\norm{\errh^{m-1}}{L^2}^2 + \frac{2 C_I^2 C_P^2 H^2 \omega^2}{C_{coer} \varepsilon^2} \norm{\errh^{m}}{L^2}^2 + \frac{36 \left((C_{inf})^2 + (C_{data}^{\prime})^2\right)^2}{C_{coer}\varepsilon^2}\norm{ \errh^{m}}{L^2}^2
\nonumber
\\
&\qquad+ \frac{2C_P^4 (C^{\prime\prime}_{data})^2}{\omega} h^2 +  2 C_I^2 C_P^4 (C^{\prime}_{data})^2 \omega h^2  + \frac{2 C^{\prime\prime\prime}_{data} (\Delta t)^2}{\omega}  + \frac{36 C_P^4 (C^{**}_{data})^2 \left((C_{inf})^2 + ( C_{data}^{\prime})^2\right)^2}{C_{coer}\varepsilon^2} h^2 
\nonumber
\\
&\qquad+ \frac{8C_P^4 (C^{\prime\prime\prime}_{data})^2}{C_{coer} \varepsilon^2}h^2 + \frac{C}{C_{coer}} \left( \frac{1}{\varepsilon^2} \left[\text{Osc}_j(\partial_t \varphi^m)\right]^2 + (C^{**}_{data})^2h^2 \right) 
\nonumber
\\
&\quad\le \frac{8}{C_{coer} \varepsilon^2 }\norm{\errh^{m-1}}{L^2}^2 + \frac{2 C_I^2 C_P^2 H^2 \omega^2}{C_{coer} \varepsilon^2} \norm{\errh^{m}}{L^2}^2 + \frac{36 \left((C_{inf})^2 + (C_{data}^{\prime})^2\right)^2}{C_{coer}\varepsilon^2}\norm{ \errh^{m}}{L^2}^2
\nonumber
\\
&\qquad+ C_{data}^{*} \left( h^2 + (\Delta t)^2 \right) + \frac{C}{C_{coer} \varepsilon^2} \left( \left[\text{Osc}_j(\partial_t \varphi^m)\right]^2 \right),
\end{align}
where $C_{data}^{*}$ depends on $\varepsilon, \omega$, etc.~but does not depend on the time step size $\Delta t$ or the mesh size $h$. Multiplying by $2\Delta t$, combining like terms and dropping a few of the positive terms on the left hand side, we arrive at
\begin{align}
&\left[1 + \omega \Delta t - 2 \Delta t\left(\frac{2 C_I^2 C_P^2 H^2 \omega^2 + 36 \left((C_{inf})^2 + (C_{data}^{\prime})^2\right)^2}{C_{coer}\varepsilon^2}\right)\right]\norm{\errh^{m}}{L^2}^2 
\nonumber
\\
&\quad \le \frac{C_{coer} \varepsilon^2 + 16 \Delta t}{C_{coer} \varepsilon^2} \norm{\errh^{m-1}}{L^2}^2 + C_{data}^{*} \Delta t \left(h^2 + (\Delta t)^2\right)+ \frac{C \Delta t}{C_{coer} \varepsilon^2} \left( \left[\text{Osc}_j(\partial_t \varphi^m)\right]^2  \right) .
\end{align}
Multiplying by $\dfrac{C_{coer} \varepsilon^2 }{C_{coer} \varepsilon^2 + 16 \Delta t}$ leads to
\begin{align}
&\left[ 1 + \Delta t \left(\frac{\omega C_{coer} \varepsilon^2 - 4 C_I^2 C_P^2 H^2 \omega^2 - 72 \left((C_{inf})^2 + ( C_{data}^{\prime})^2\right)^2 - 16}{C_{coer} \varepsilon^2  + 16 \Delta t }\right) \right] \norm{\errh^{m}}{L^2}^2 
\nonumber
\\
&\quad \le \norm{\errh^{m-1}}{L^2}^2 + \dfrac{ C_{coer} \varepsilon^2 \Delta t}{(C_{coer} \varepsilon^2 + 16 \Delta t)} \left( h^2 + h^5 + (\Delta t)^2\right)C_{data}^{*},
\end{align}
where we have used the bound $\text{Osc}_j(\partial_{t} \varphi^m) \le C h^{2 + \min\{\ell,k+1\}}$ such that $k=2$ is the polynomial degree of the finite element space $Z_h$ and $\partial_t \varphi \in L^\infty(0,\infty;H^{3}(\Omega))$ ($\ell = 3$) by the higher regularity \eqref{eq:higher-regularities} assumption to obtain a bound on the oscillations of $\partial_t \varphi$. (See \cite{brenner:11:frontiers} for details.)

Lemma \ref{lem:geometric-series-bound} now guarantees the required results provided the choice of $H$ and $\omega$ allow for 
\begin{align*}
\lambda_1 := \dfrac{ C_{coer} \varepsilon^2\omega - 4 C_I^2 C_P^2 H^2 \omega^2 - 72 \left((C_{inf})^2 + (C_{data}^{\prime})^2\right)^2 - 16}{C_{coer} \varepsilon^2  + 16 \Delta t } > 0.
\end{align*} 
to hold.
%Thus, we require
%\begin{align*}
%4 C_I^2 C_P^2 H^2 \omega^2 -  C_{coer} \varepsilon^2\omega + 72 \left(({\color{black} C_{inf} })^2 + (C_{data}^{**})^2\right)^2 + 16 < 0
%\end{align*}
%which is satisfied under the conditions that 
%\begin{align}
%0 < \frac{C_{coer} \varepsilon^2 - \beta}{8 C_I^2 C_P^2 H^2} < \omega < \frac{C_{coer} \varepsilon^2 - \beta}{8 C_I^2 C_P^2 H^2} 
%\\
%\text{and}
%\nonumber
%\\
%H \le \frac{C_{coer} \varepsilon^2}{4 C_I C_P \sqrt{16 + 72 \left(({\color{black} C_{inf} })^2 + (C_{data}^{**})^2\right)^2}},
%\end{align}
%where $\beta^2 = (C_{coer} \varepsilon^2)^2 - 16 C_I^2 C_P^2 H^2 \left(16 + 72 \left(({\color{black} C_{inf} })^2 + (C_{data}^{**})^2\right)^2\right)$.
\end{proof}

\section{Numerical Experiments}\label{sec:numerical-experiments}

In this section, we present results of several numerical experiments which demonstrate the effectiveness of the proposed data assimilation finite element method. The Firedrake Project \cite{R:16:firedrake} was used to perform all numerical experiments. We use a square domain $\Omega = (0,1)^2$ %\subset \mathbb{R}^2$ 
and take $\TTh$ to be a regular triangulation of $\Omega$ consisting of right isosceles triangles which is a quasi-uniform family. (We use a family of meshes ${\mathcal T}_h$ such that no triangle in the mesh has more than one edge on the boundary.) Additionally, in each experiment, we set the interfacial width parameter $\varepsilon = 0.05$.

%\begin{figure}[h!]
%\centering
%\includegraphics[scale=.15]{H_opt_3_points_mesh_actual}
%\caption{A mesh with $h = \nicefrac{\sqrt{2}}{64}$ on a domain $\Omega = (0,1)^2\subset \mathbb{R}^2$ where the black dots indicate the points corresponding to a data assimilation grid with $H = 0.03125$. Note that there are $32 \times 32 = 1024$ data assimilation grid points.}
%\label{fig:mesh-grid}
%\end{figure}
%

The data assimilation term $\omega \iprd{I_H(\phih^m - \varphi^m)}{\psi}$ was computed as follows. A true solution $\varphi$ was obtained at all times by selecting a cross shaped region as initial conditions as shown in the top right image of Figure \ref{fig:cross-exp}, setting the nudging parameter $\omega = 0$, and solving the CH equation using the C$^0$ interior penalty FEM \eqref{eq:fully-discrete-fem}. A data assimilation grid size $H$ was chosen and grid points were identified and located on the finite element mesh. A vector was then created such that the value of 1 was assigned for all nodes corresponding to these grid points and a value of 0 was assigned for all other nodes. Let us name this vector $v$. Then the data assimilation term $\omega \iprd{I_H(\phih^m - \varphi^m)}{\psi}$ was computed by
\begin{align*}
\omega \iprd{I_H(\phih^m - \varphi^m)}{\psi} = \omega \iprd{v \phih^m - v \varphi^m}{v\psi},
\end{align*}
where we note that $v \in Z_h$ and that this is equivalent to the interpolation method onto a coarse mesh of piecewise constants $Z_H$, as described in \cite{RZ19}.  Finally, in each of the experiments, the initial conditions for the numerical solution $\phi_0$ was set to random initial conditions as shown in Figure \ref{fig:cross-exp}.
%Figure \ref{fig:mesh-grid} illustrates the chosen mesh with $h=\nicefrac{\sqrt{2}}{64}$ and chosen grid points with $H = 0.03125$ on a domain $\Omega = (0,1)^2\subset \mathbb{R}^2$.  We %note that since a $\mathcal{P}^2$ Lagrange finite element space is chosen, then a mesh size of $h=\nicefrac{\sqrt{2}}{64}$ has $33,025$ nodes and a grid size of $H = 0.03125$ corresponds to %$1024$ points or approximately 3\% of those nodes. 

The first numerical experiment demonstrates the effectiveness of the CDA-FEM for various grid sizes $H$. For this experiment, we set the nudging parameter $\omega = \nicefrac{1}{\varepsilon^2} = 400$ as indicated by the theory above. We then chose five different grid sizes $H = 0.011049, 0.015625, 0.03125, $ $0.0625,$ and $0.125$, which correspond respectively to 8,100, 4,096, 1,024, 256 and 64 grid points, while the fine mesh uses piecewise quadratics and has 33,025 grid points.   Theorem \ref{thm:main-error-result} provides a sufficient condition that the grid size should be chosen as $H = \mathcal{O}( \varepsilon^2)$ but our experiments suggest that a grid size much coarser than that will produce good results. 

\begin{figure}[h!]
\subfloat{\includegraphics[width=0.45\textwidth]{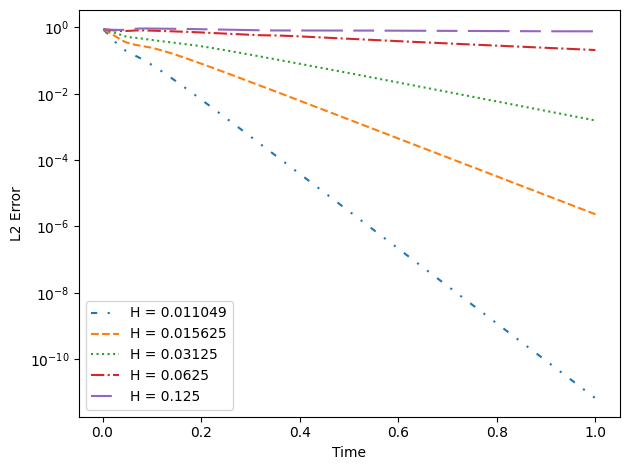}}
\subfloat{\includegraphics[width=0.45\textwidth]{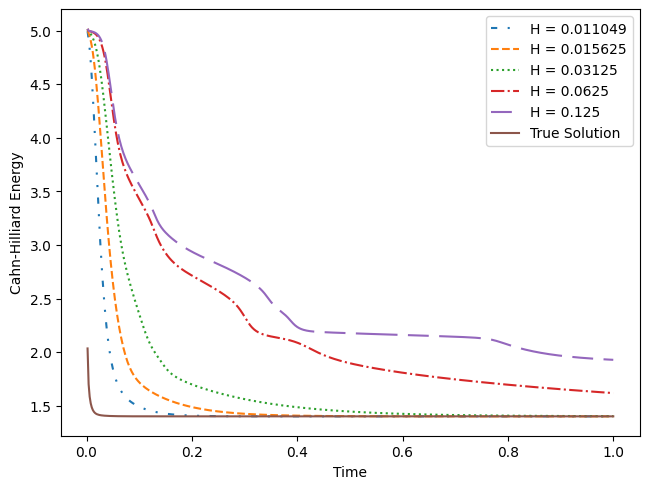}}
\caption{The $L2$ error between the true solution and the solution to the data assimilation finite element method for various data assimilation grid sizes is shown on the left and energy dissipation with respect to the CH energy is shown on the right. The mesh size is $h = \nicefrac{\sqrt{2}}{64}$ and the time step size is $\Delta t = 0.002$. All other parameters are defined in the text. }
\label{fig:grid-error-energy}
\end{figure}

Figure \ref{fig:grid-error-energy} shows a semi-log plot of the error between the true solution and the solution to the CDA-FEM \eqref{eq:fully-discrete-fem} measured in the $L^2$ norm for the five different grid sizes on the left. All but the coarsest grid size of $H = 0.125$ converge exponentially with respect to time to the true solution.  To verify that a grid size of $H = 0.0625$ converges as expected but a grid size of $H = 0.125$ does not, we additionally show re-scaled semi-log plots of the error for these two grid sizes in Figure \ref{fig:grid-error-zoom}. However, it is interesting to note that the grid size of $H = 0.125$ does look like it may eventually converge to the true solution. Additionally, if solutions to the CDA-FEM are converging to the true solution, one would expect that the CH energy of solutions to the CDA-FEM would converge to the CH energy of the true solution. We illustrate that this is the case for the grid sizes $H = 0.011049, 0.015625,$ and $0.03125$ in the image on the right of Figure \ref{fig:grid-error-energy}. 

\begin{figure}[h!]
\subfloat{\includegraphics[width=0.45\textwidth]{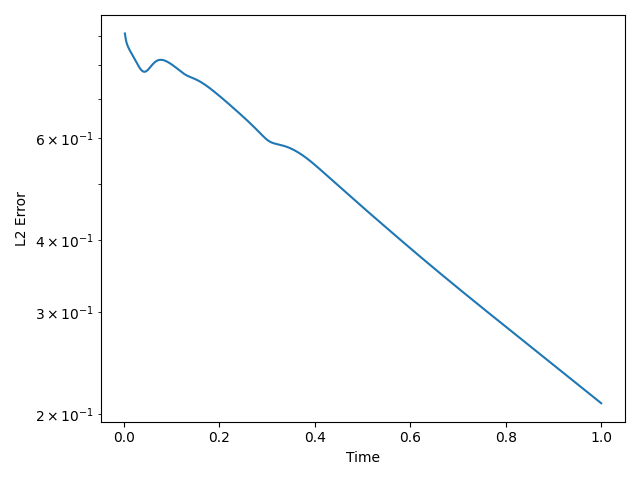}}
\subfloat{\includegraphics[width=0.45\textwidth]{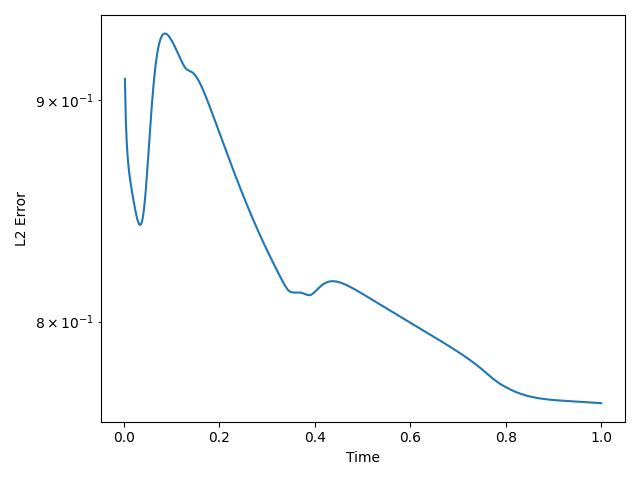}}
\caption{The $L^2$ error between the true solution and the solution to the data assimilation finite element method with a data assimilation grid size of $H = 0.0625$ on the left and $H = 0.125$ on the right. The mesh size is $h = \nicefrac{\sqrt{2}}{64}$ and the time step size is $\Delta t = 0.002$. All other parameters are defined in the text. }
\label{fig:grid-error-zoom}
\end{figure}

The second numerical experiment demonstrates the effectiveness of the CDA-FEM for various values of the nudging parameter $\omega$. For this experiment, we set the data assimilation grid to be $H = 0.03125,$ and chose five different values for the nudging parameter $\omega = 1, 20, 400, 1000,$ and $5000$. Theorem \ref{thm:main-error-result} admits a sufficient condition that the appropriate value for the nudging parameter $\omega$ is at least $\nicefrac{1}{\varepsilon^2} = 400$, but if $\omega$ is too large then $H$ needs to be very small.  However, our experiments show that good results can also be obtained for much larger values of $\omega$.  Figure \ref{fig:nudge-error-energy} shows a semi-log plot of the error between the true solution and the solution to the CDA-FEM \eqref{eq:fully-discrete-fem} measured in the $L^2$ norm for the five different values of the nudging parameter. Only values of $\omega \ge \nicefrac{1}{\varepsilon^2} = 400$ converge exponentially with respect to time as expected. To verify that values of the nudging parameter $\omega = 1$ and $\omega = \nicefrac{1}{\varepsilon} = 20$ do not converge as expected, we additionally show re-scaled semi-log plots of the error for these values of the nudging parameter in Figure \ref{fig:nudge-error-zoom}. One might also expect that increasing the nudging parameter above $\omega = 5000$ will only improve the results. However, we note that in this case, the linear solver may break down.  Convergence of energy for the simulations is also shown in Figure \ref{fig:nudge-error-energy}, and we observe that the simulations that converged to the true solution in $L^2$ norm also found the correct energy, while those that did not converge ($\omega=20,1$) did not find the correct energy.

%Again, if solutions to the data assimilation finite element method are truly converging to the true solution, one would expect that the CH energy of solutions to the data assimilation finite element %method would converge to the CH energy of the true solution. We illustrate that this is the case for the grid sizes $\omega = 400, 1000,$ and $5000$ in the image on the right of Figure %\ref{fig:nudge-error-energy}. 

\begin{figure}[h!]
\subfloat{\includegraphics[width=0.45\textwidth]{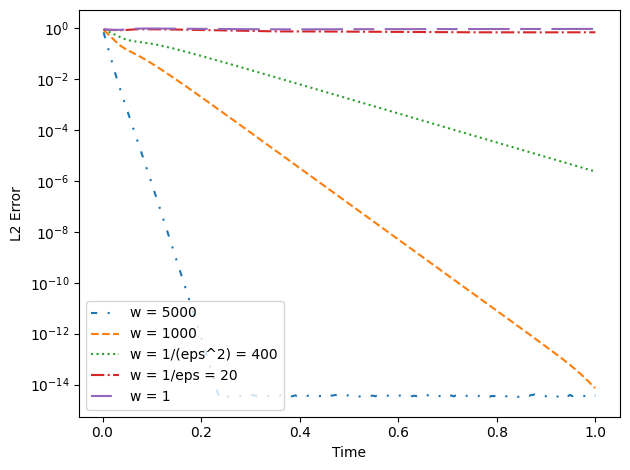}}
\subfloat{\includegraphics[width=0.45\textwidth]{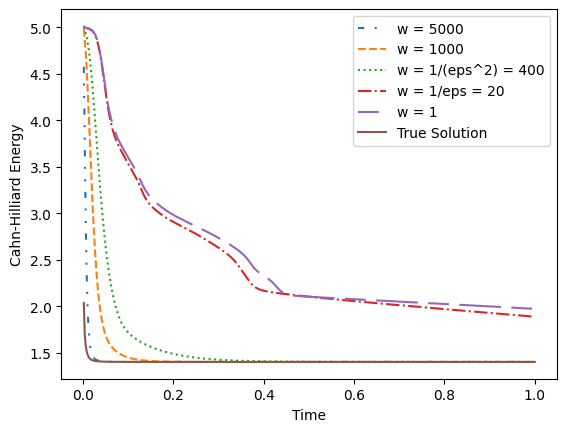}}
\caption{The $L^2$ error between the true solution and the solution to the CDA-FEM for various values of the nudging parameter on the left and energy dissipation with respect to the CH energy is shown on the right.  The mesh size is $h = \nicefrac{\sqrt{2}}{64}$ and the time step size is $\Delta t = 0.002$. All other parameters are defined in the text. }
\label{fig:nudge-error-energy}
\end{figure}

\begin{figure}[h!]
\subfloat{\includegraphics[width=0.45\textwidth]{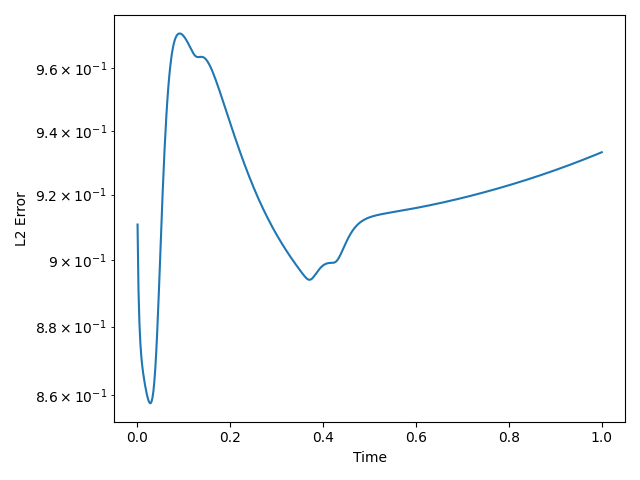}}
\subfloat{\includegraphics[width=0.45\textwidth]{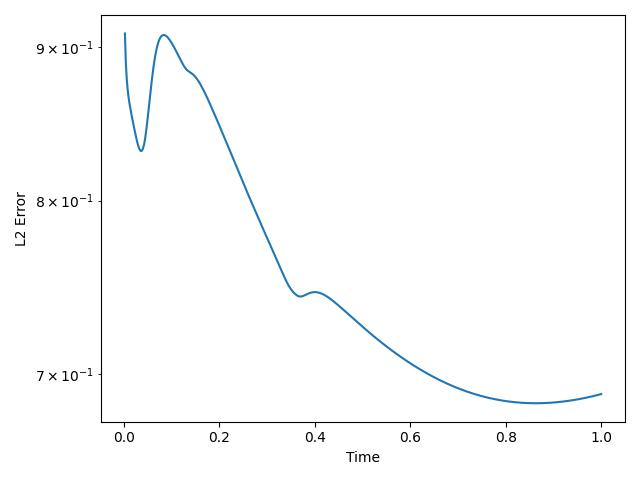}}
\caption{The $L^2$ error between the true solution and the solution to the CDA-FEM with a nudging parameter of $\omega = 1$ on the left and $\omega = 20$ on the right. The mesh size is $h = \nicefrac{\sqrt{2}}{64}$ and the time step size is $\Delta t = 0.002$. All other parameters are defined in the text. }
\label{fig:nudge-error-zoom}
\end{figure}

In viewing the results of the first two experiments above, the performance of the CDA-FEM \eqref{eq:fully-discrete-fem} appears to be more sensitive to the value of the nudging parameter $\omega$ than the data assimilation grid size $H$. To determine if setting a higher value for the nudging parameter can overcome the deficiencies seen by taking coarse grid sizes, we repeated the first experiment with a nudging parameter set equal to $\omega = 1000$. Figure \ref{fig:grid-error-1000} illustrates that increasing the nudging parameter does help improve the results if a coarse grid size $H$ is chosen. This is best illustrated by comparing the convergence of $H = 0.0625$ shown in Figure \ref{fig:grid-error-1000} to that shown in Figure \ref{fig:grid-error-energy}, although all but the grid size $H=0.125$ show dramatic improvement.

\begin{figure}[h!]
\centering
\includegraphics[width=0.45\textwidth]{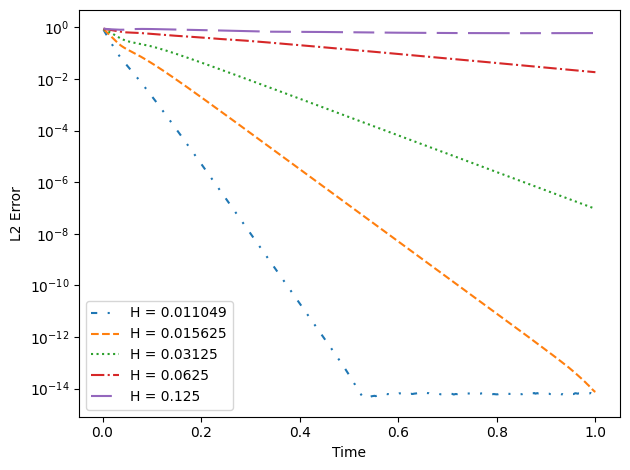}
\caption{The $L2$ error between the true solution and the solution to the data assimilation finite element method for various data assimilation grid sizes. The mesh size is $h = \nicefrac{\sqrt{2}}{64}$ and the time step size is $\Delta t = 0.002$. All other parameters are defined in the text. }
\label{fig:grid-error-1000}
\end{figure}

Finally, in Figure \ref{fig:cross-exp}, we present images of the true solution with initial conditions set as a cross shaped region, the solution to the CDA-FEM \eqref{eq:fully-discrete-fem} with random initial conditions and a nudging parameter of $\omega = 400$ with a data assimilation grid size of $H = 0.03125$, and solutions to the finite element method \eqref{eq:fully-discrete-fem} with random initial conditions and a nudging parameter of $\omega = 0$ side by side at times $t =0.0, 0.002, 0.01, 0.05,1.0 $. A mesh size of $h = \nicefrac{\sqrt{2}}{64}$ and an interfacial width parameter of $\varepsilon = 0.05$ was chosen for each.  Convergence of the CDA-FEM to the true solution is observed in the sequence of plots, while the solution without data assimilation finds a different long time steady state.

\begin{figure}[h!]
\centering
\subfloat	{\includegraphics[width=0.25\textwidth]{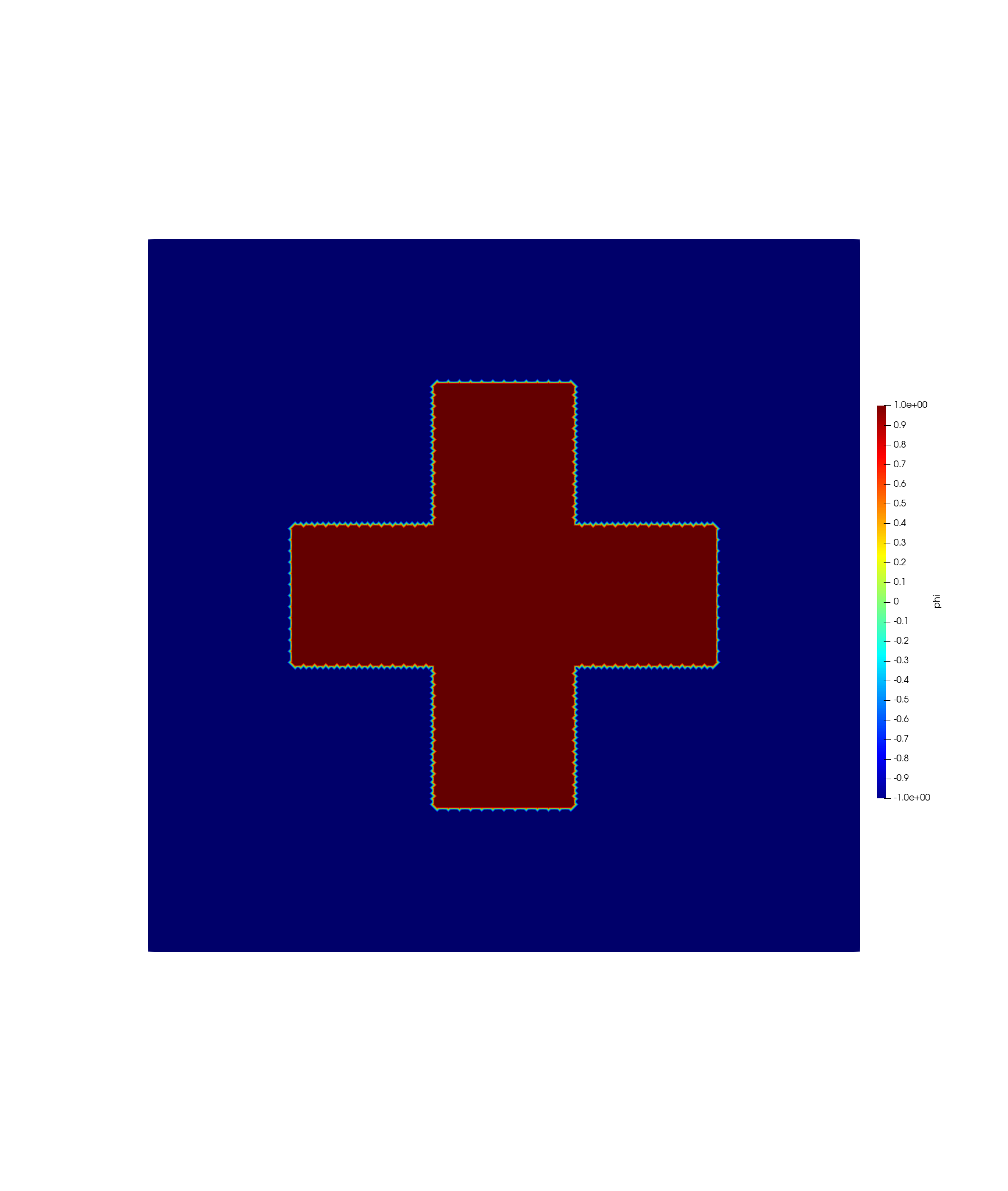}}
\subfloat	{\includegraphics[width=0.25\textwidth]{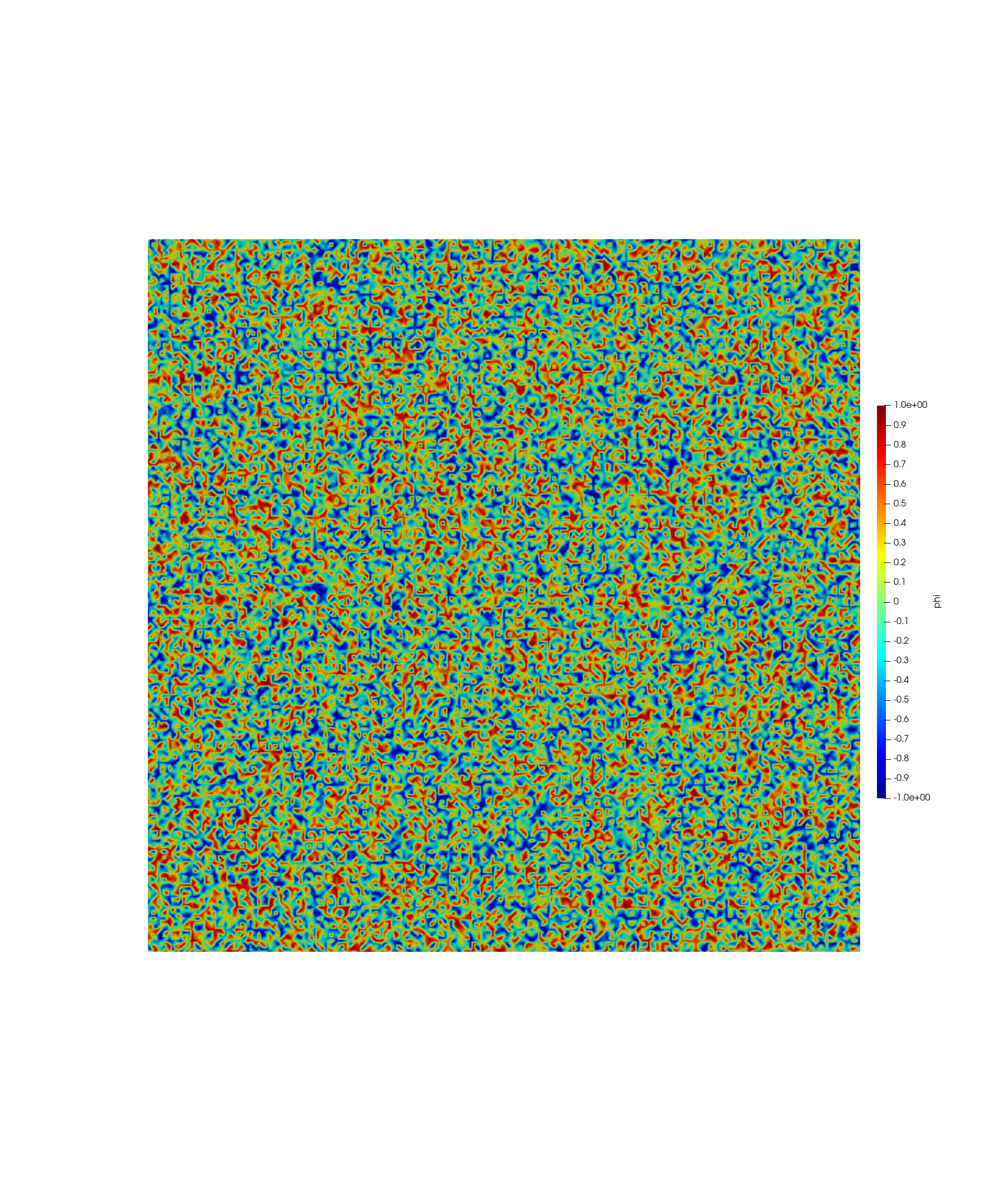}}
\subfloat	{\includegraphics[width=0.25\textwidth]{rand_init_conds}}
\\[-8ex]
\subfloat	{\includegraphics[width=0.25\textwidth]{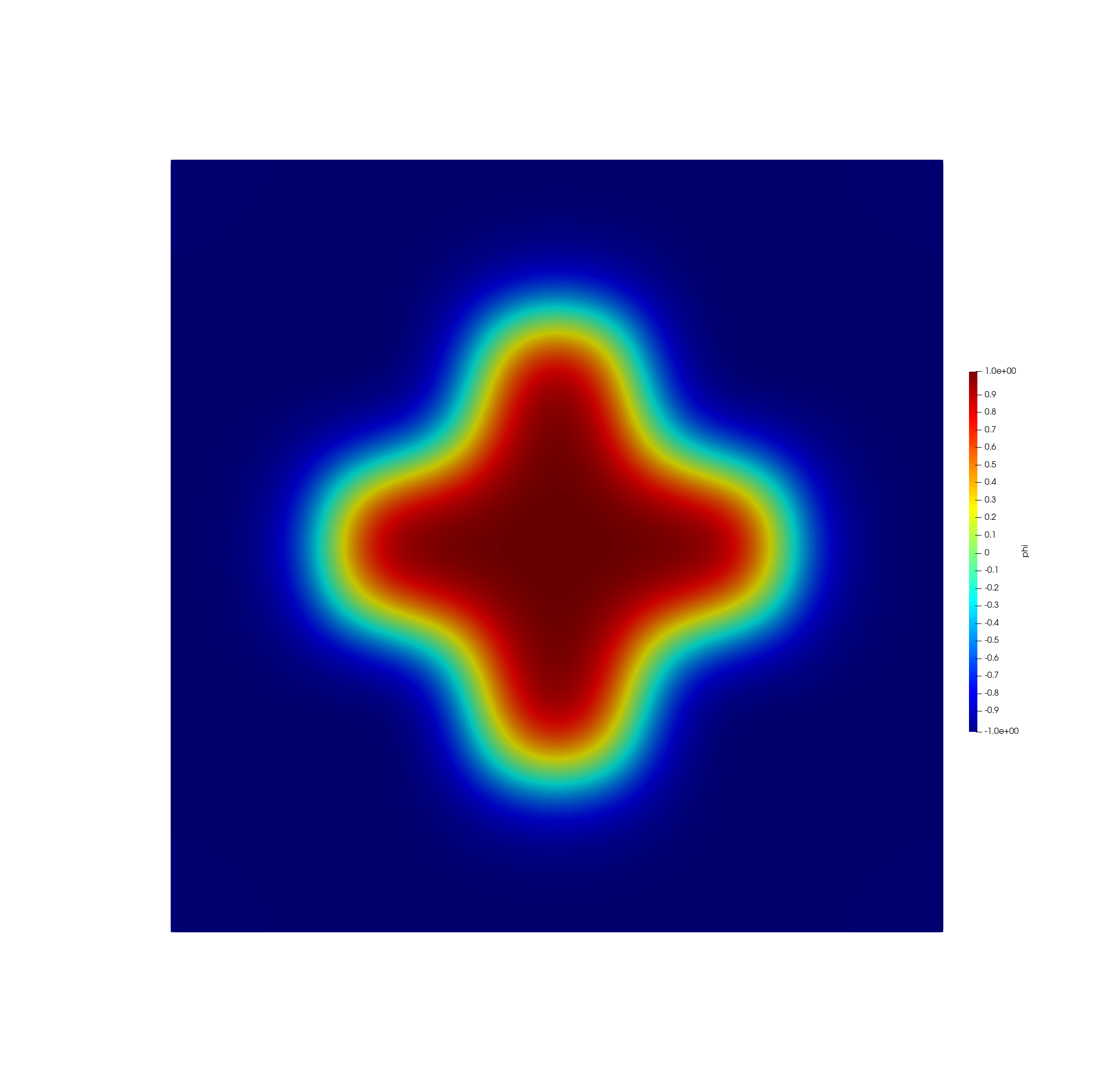}}
\subfloat	{\includegraphics[width=0.25\textwidth]{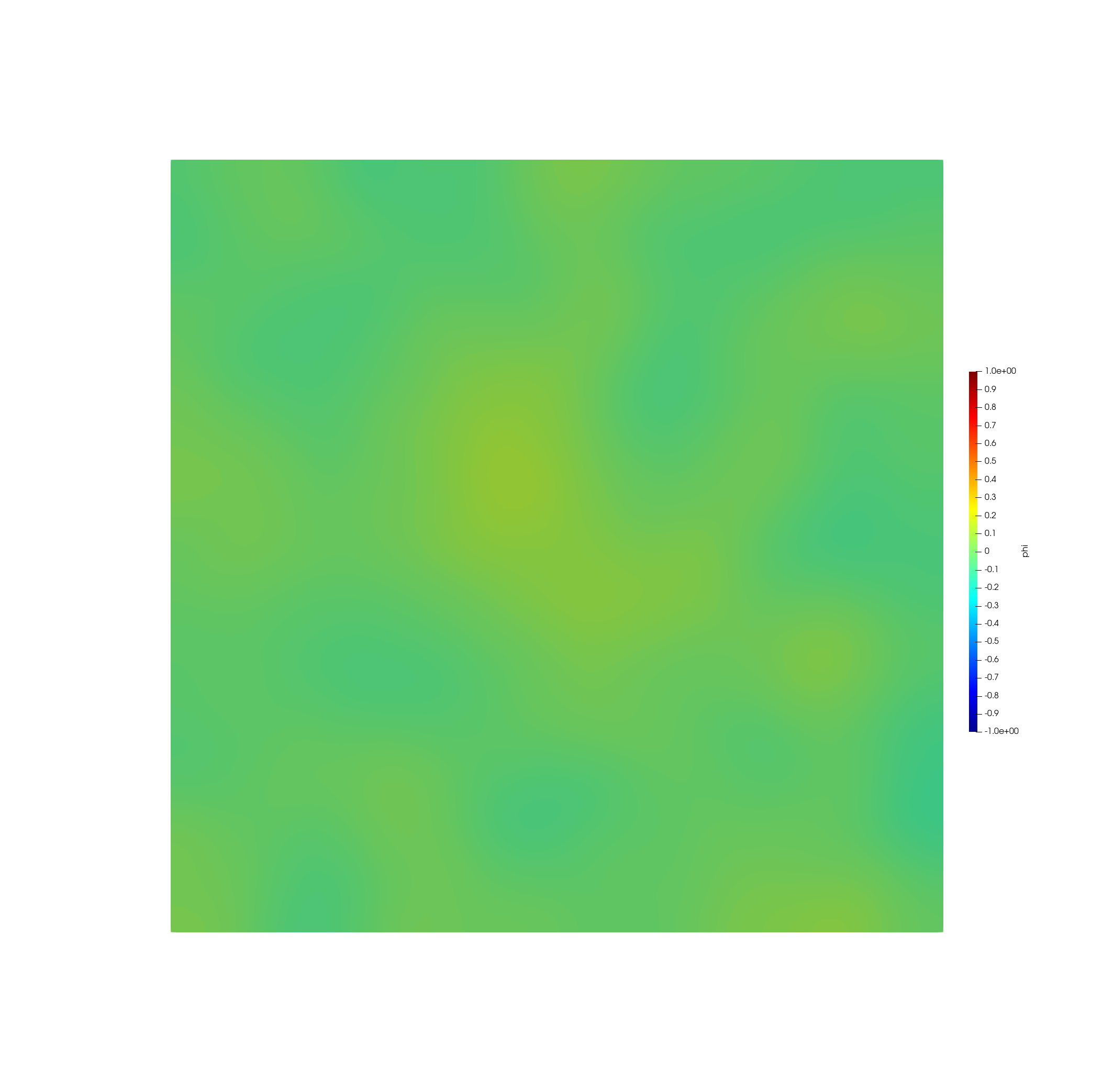}}
\subfloat	{\includegraphics[width=0.25\textwidth]{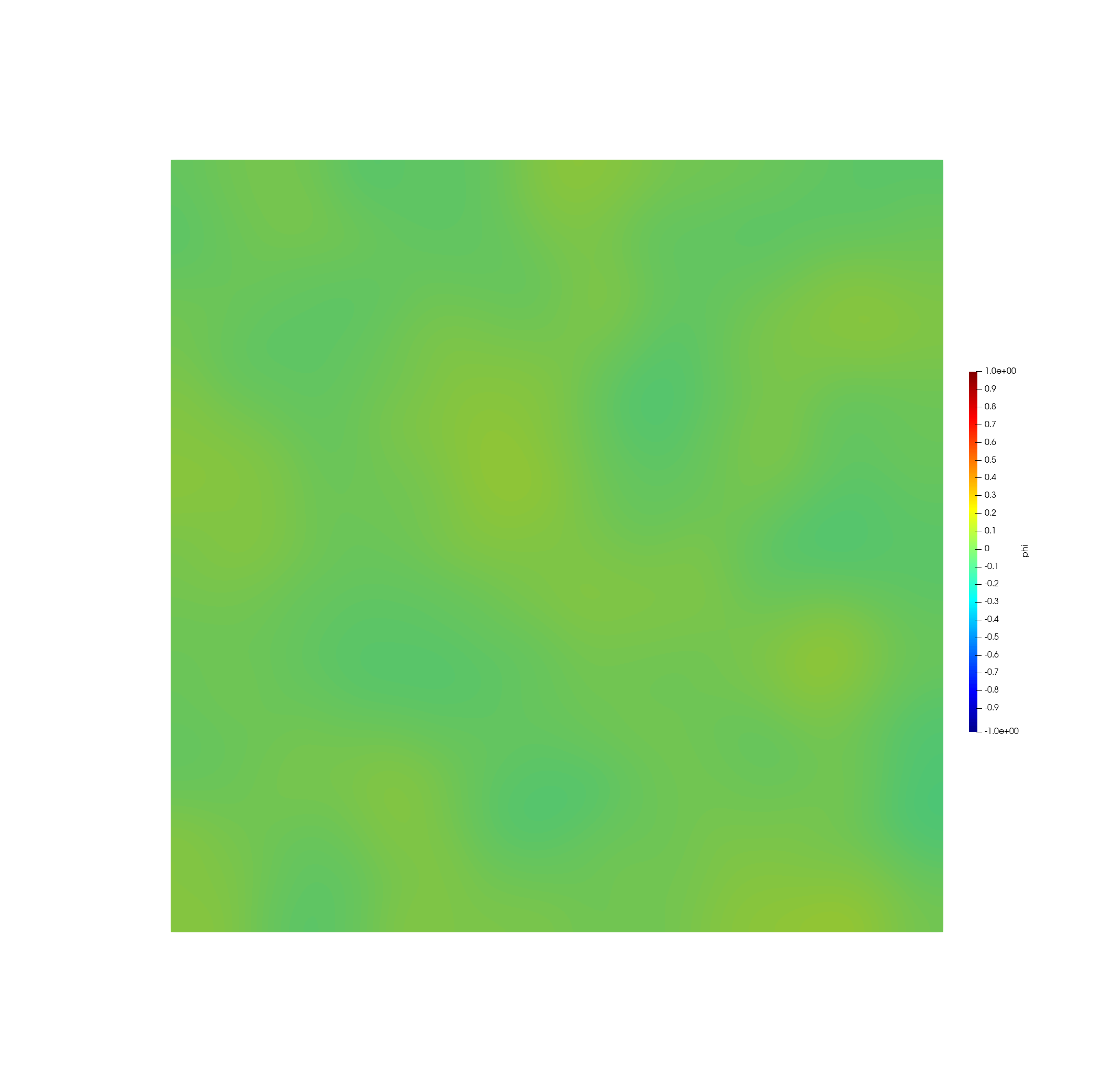}}
\\[-5ex]
\subfloat	{\includegraphics[width=0.25\textwidth]{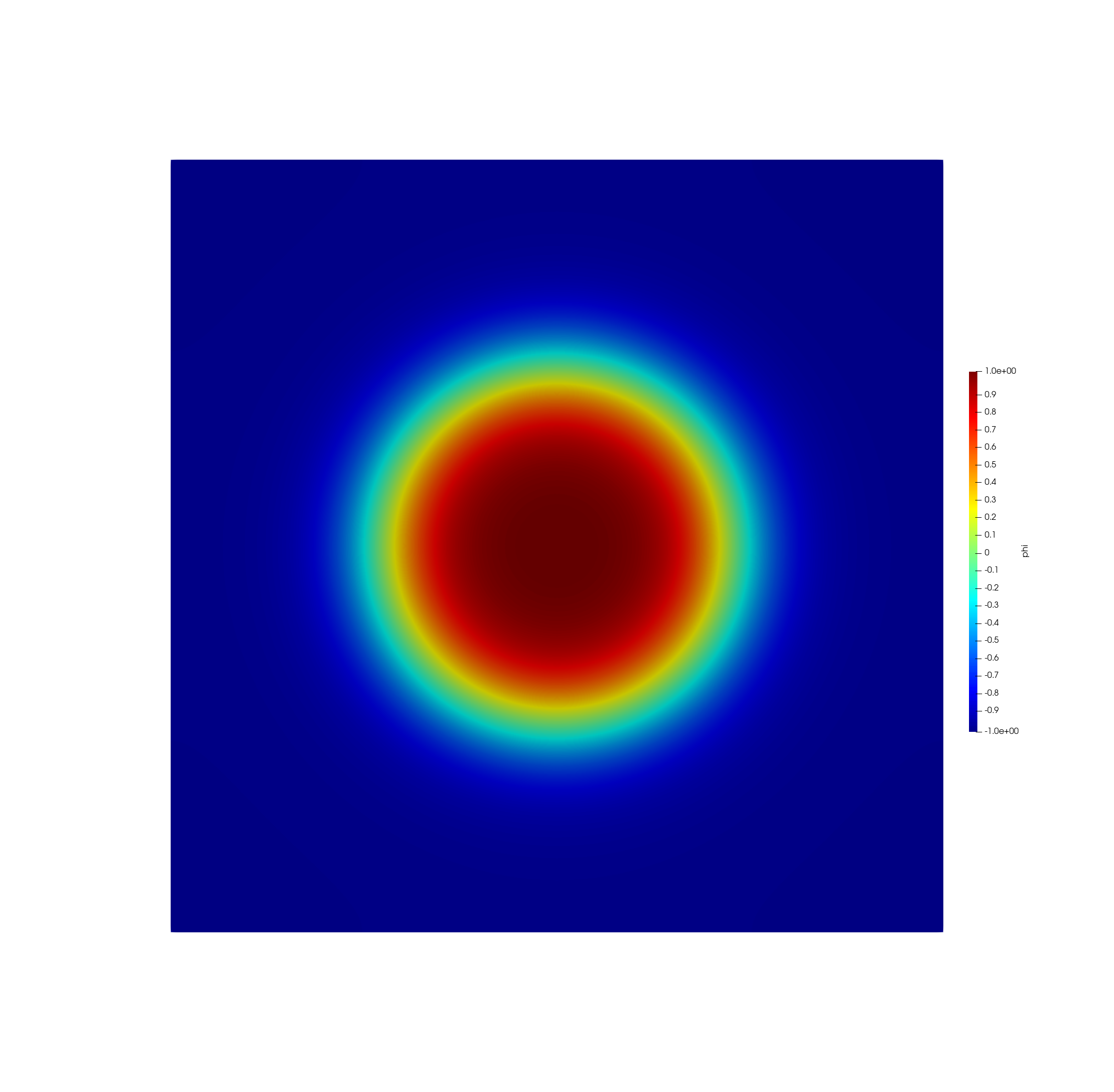}}
\subfloat	{\includegraphics[width=0.25\textwidth]{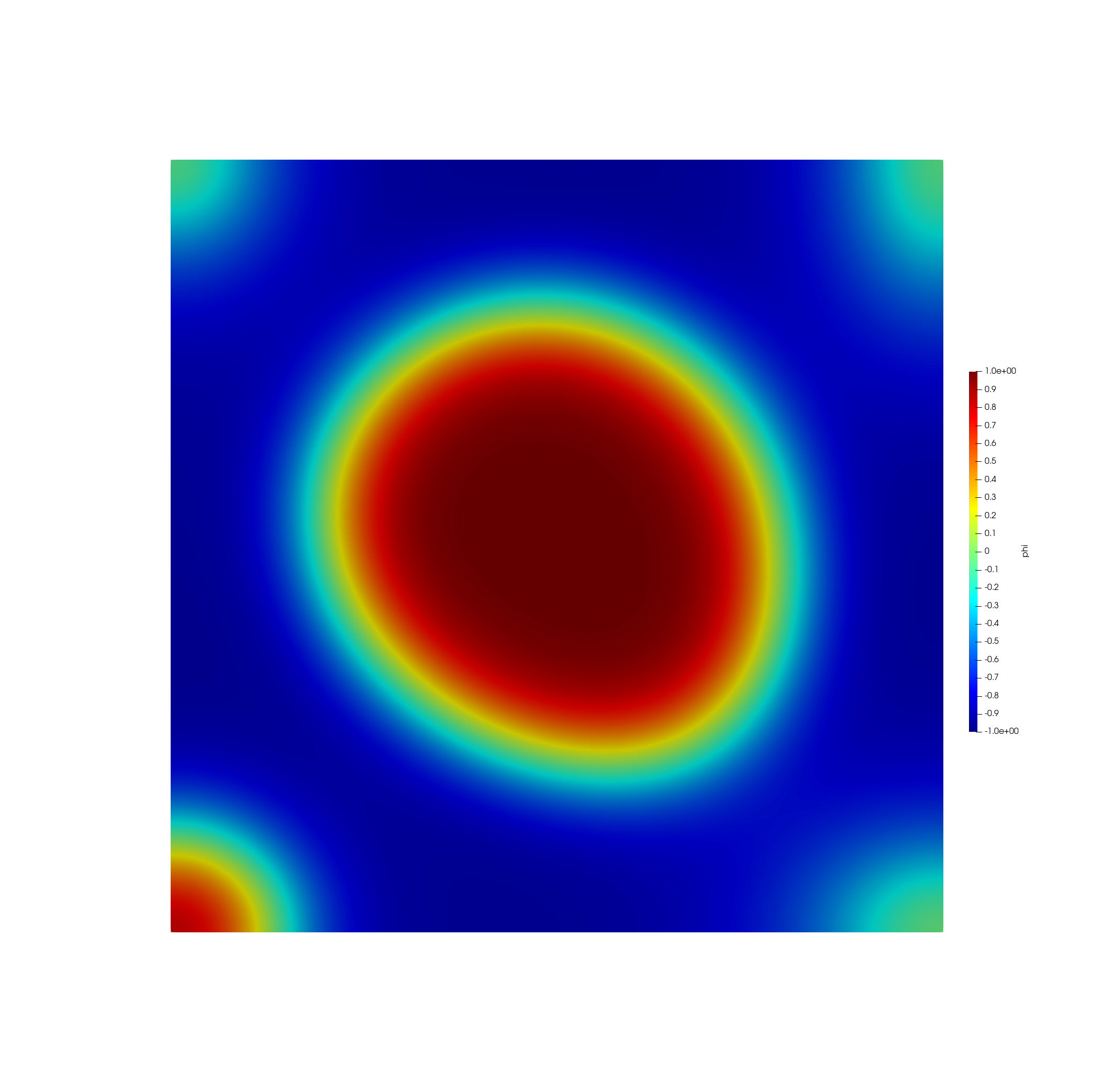}}
\subfloat	{\includegraphics[width=0.25\textwidth]{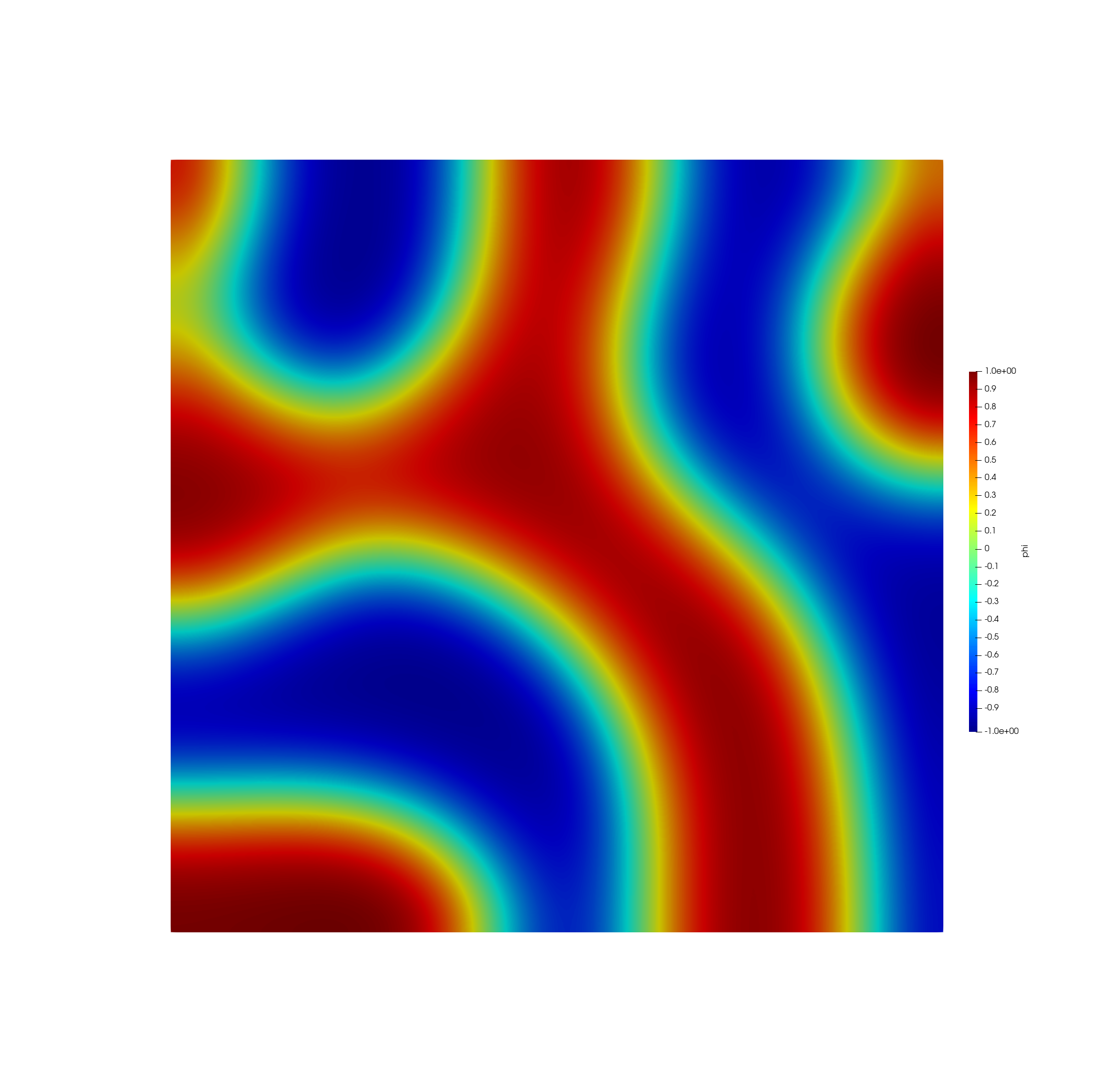}}
\\[-5ex]
\subfloat	{\includegraphics[width=0.25\textwidth]{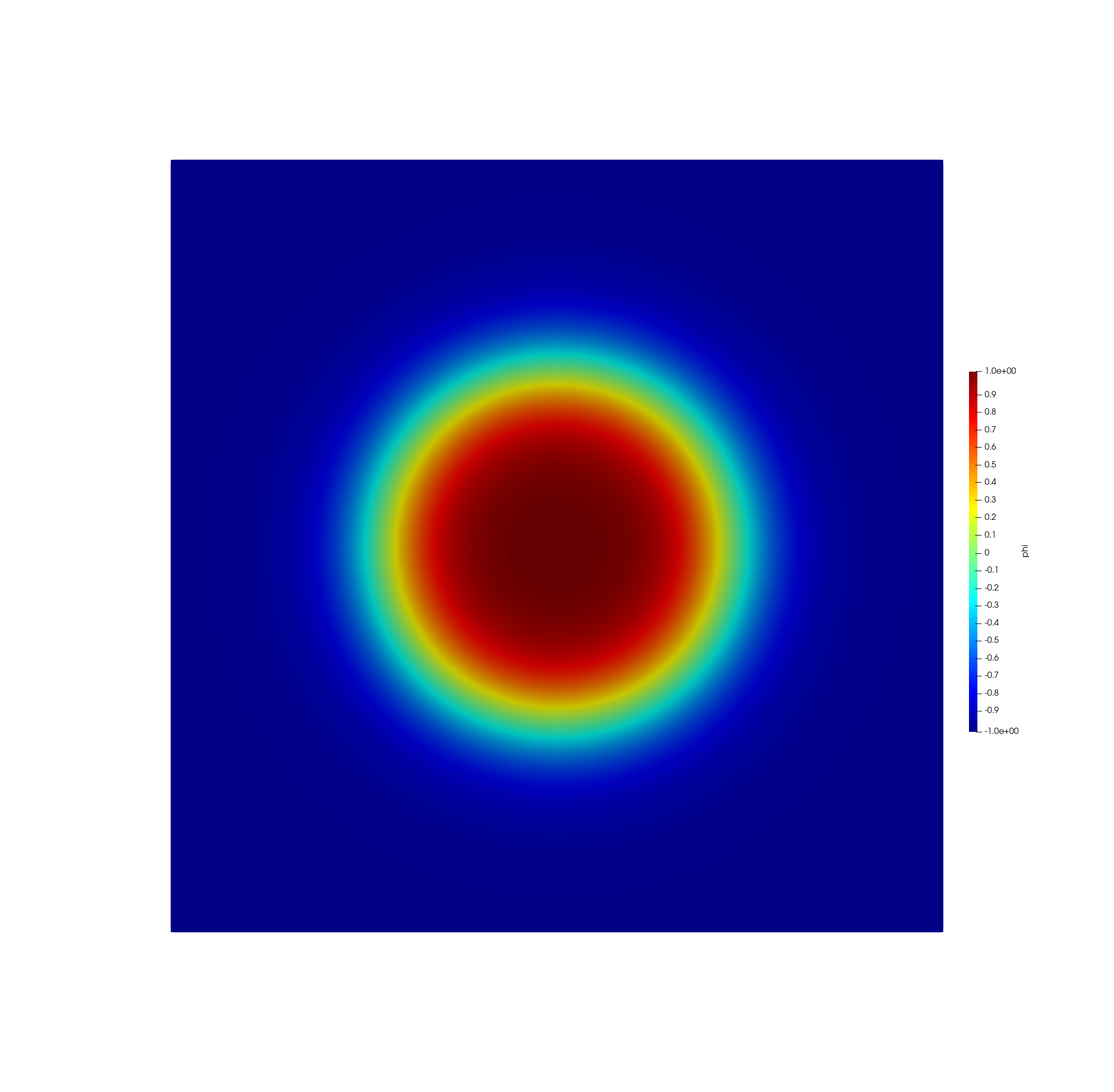}}
\subfloat	{\includegraphics[width=0.25\textwidth]{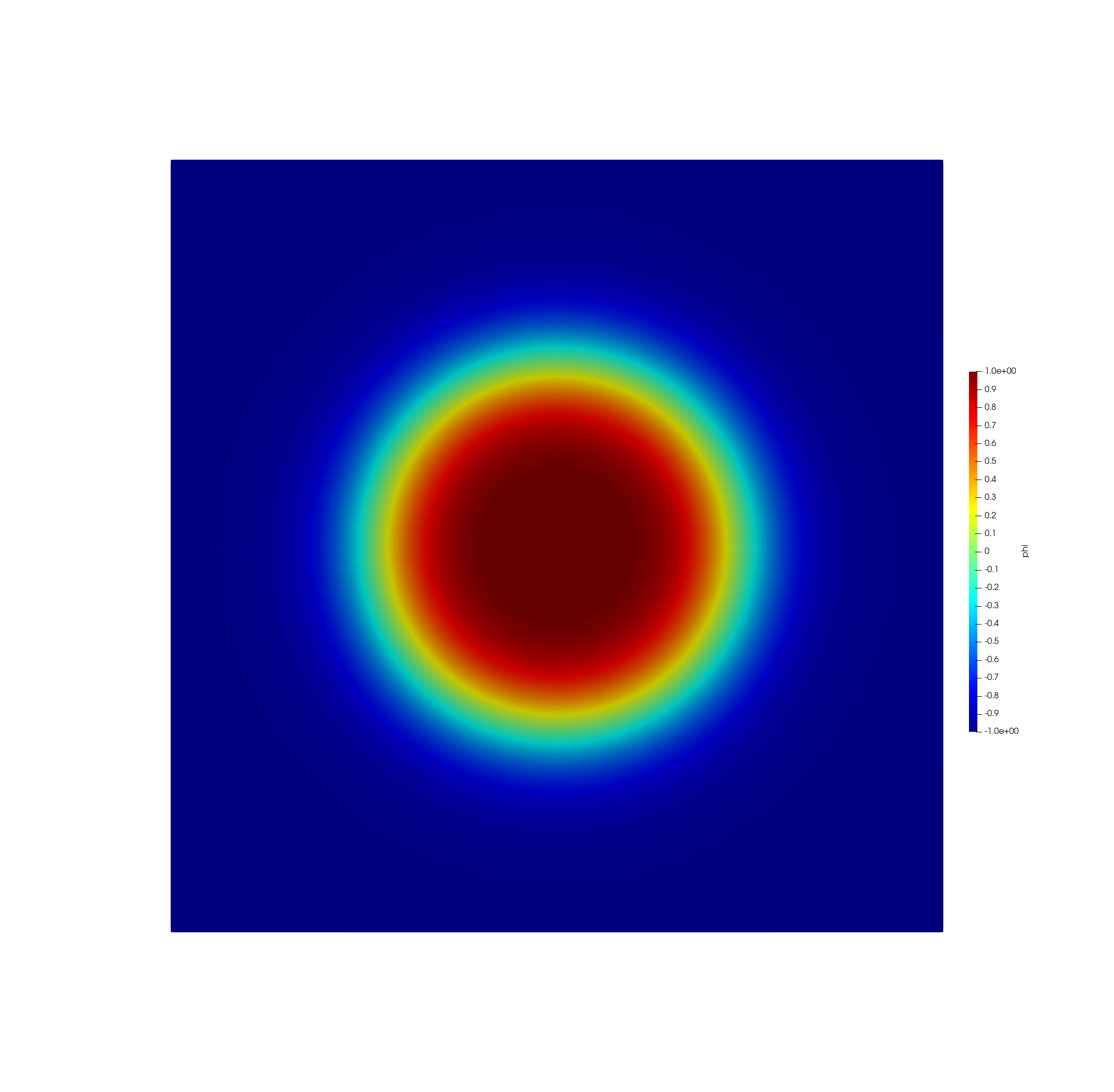}}
\subfloat	{\includegraphics[width=0.25\textwidth]{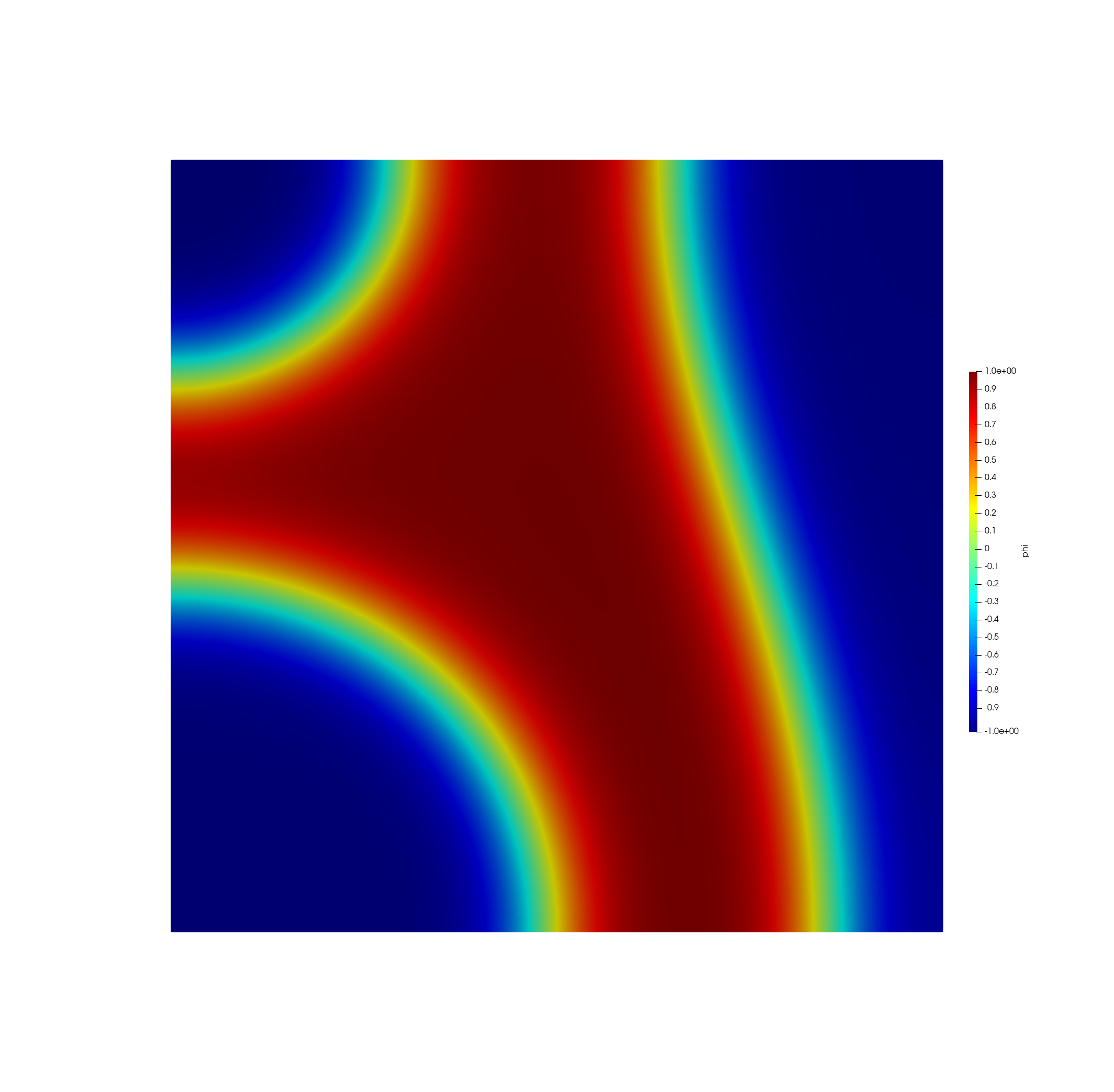}}
\\[-5ex]
\subfloat	{\includegraphics[width=0.25\textwidth]{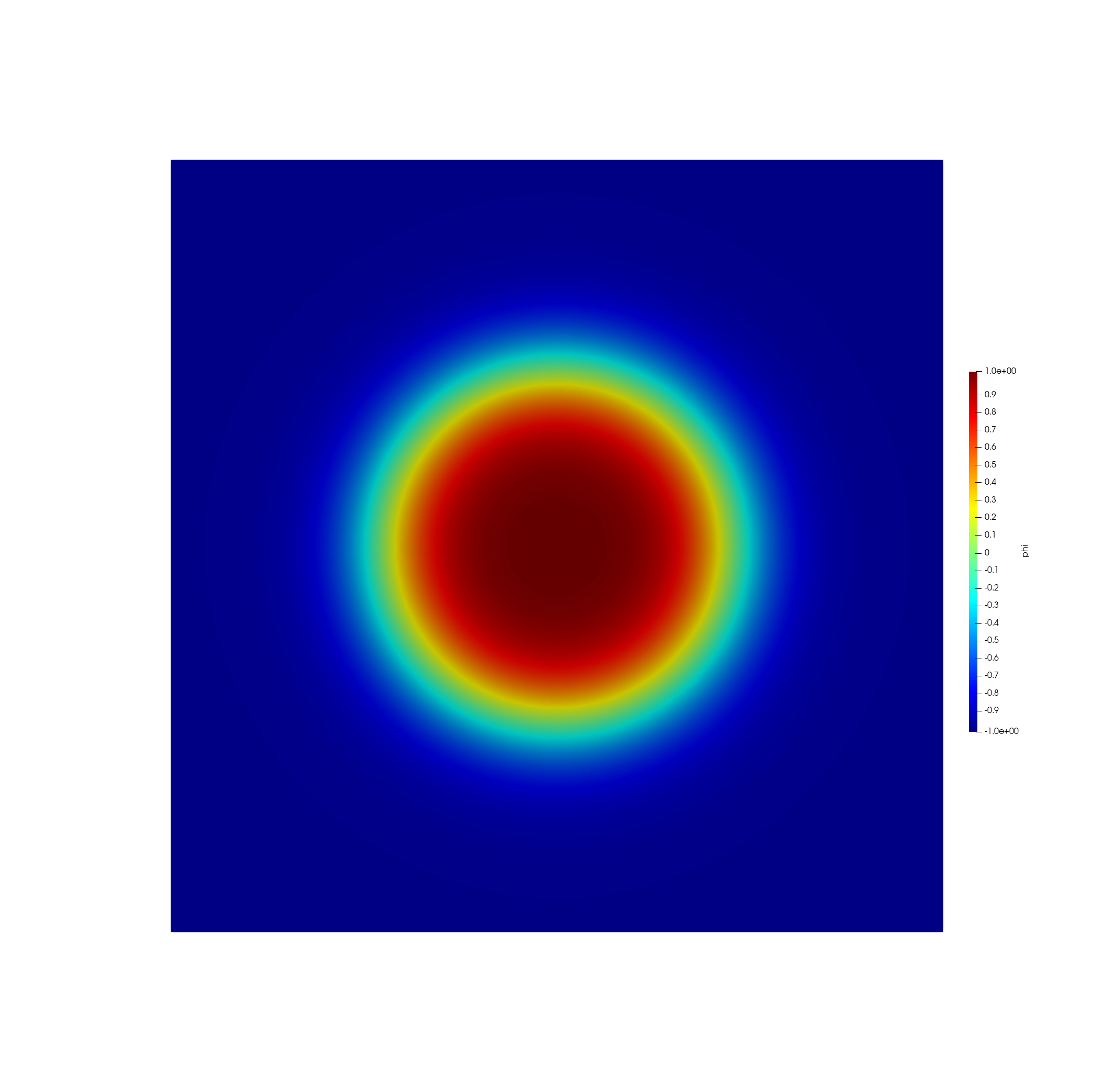}}
\subfloat	{\includegraphics[width=0.25\textwidth]{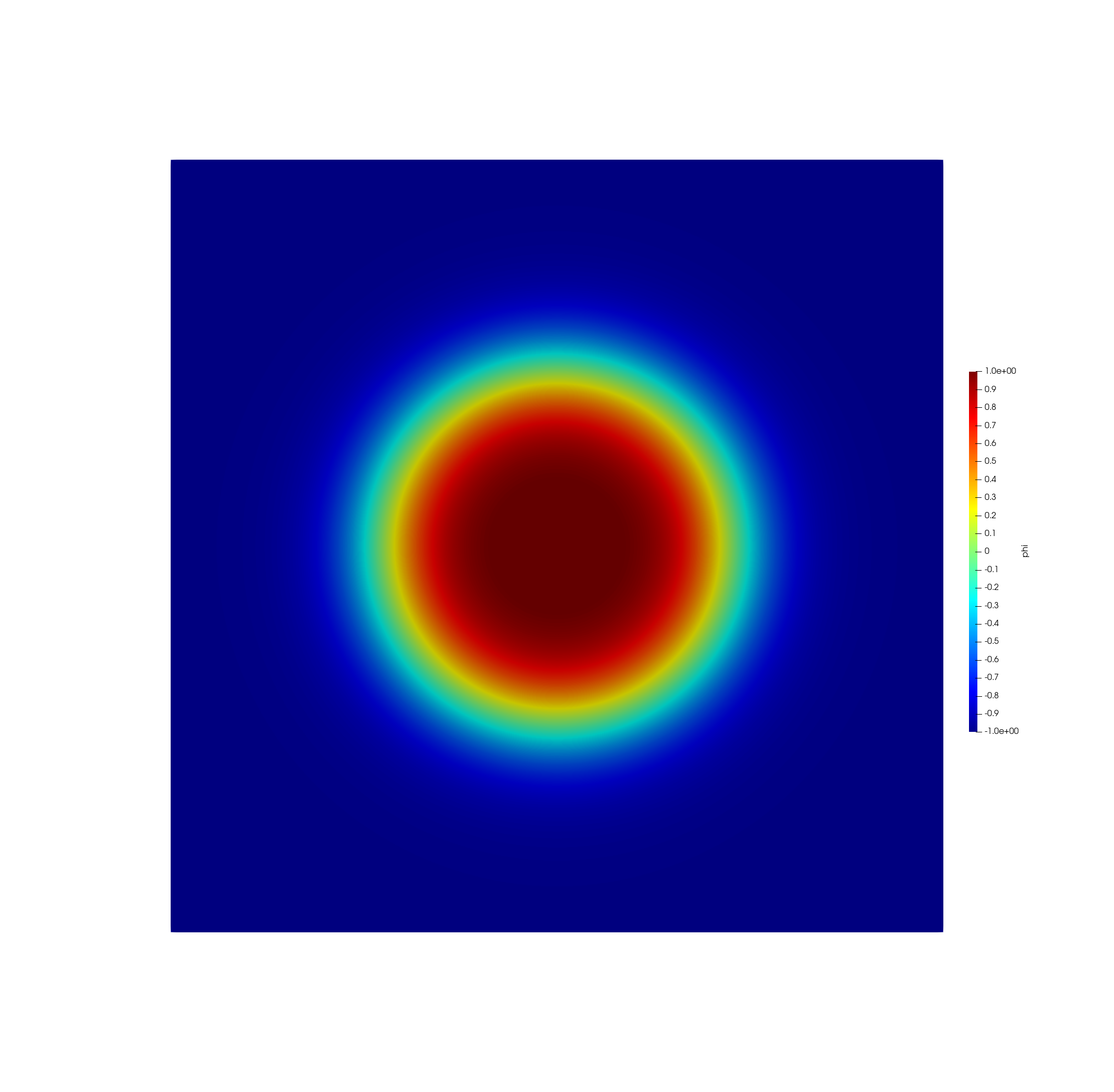}}
\subfloat	{\includegraphics[width=0.25\textwidth]{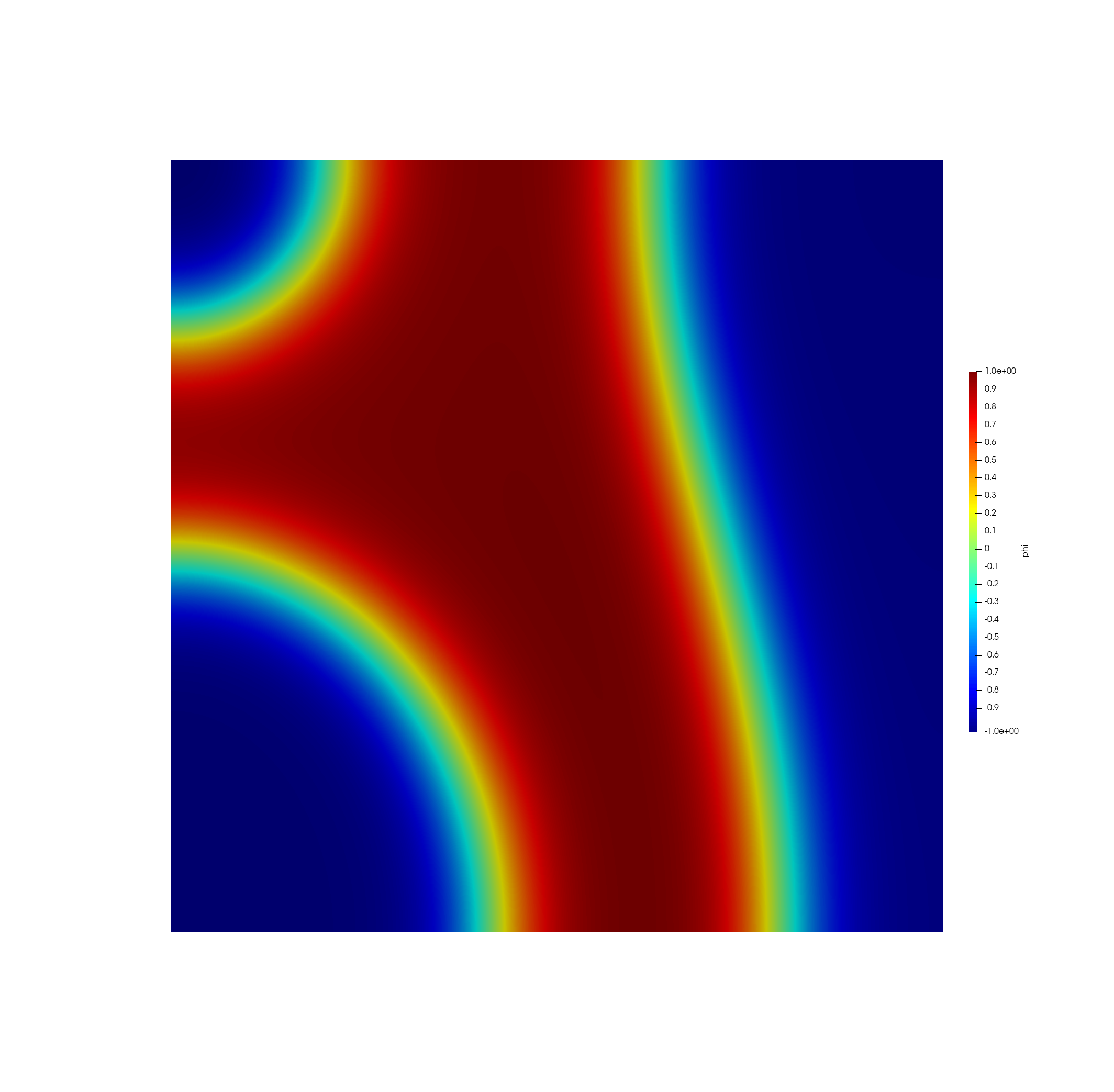}}
\caption{Comparison of the true solution (left), the data assimilation finite element method with random initial conditions and $\omega = \varepsilon^{-2}$ (middle), and the data assimilation finite element method with random initial conditions and $\omega = 0$ (right). The times on each line from top to bottom are $t=0.0, 0.002, 0.01, 0.05,1.0 $. The other parameters are as follows: $\Delta t = 0.002, h = \nicefrac{\sqrt{2}}{64}, H = 0.03125$.}
\label{fig:cross-exp}
\end{figure}

\section{Conclusions and Future Directions}

We proposed, analyzed and tested a CDA-FEM method for the Cahn-Hilliard equations.  A fourth order formulation of Cahn-Hilliard was used, and so that common FE software packages could be used, the spatial discretization used was C$^0$ interior penalty.  We proved long time stability and accuracy of the method, provided enough measurement points and a large enough nudging parameter.  Numerical tests revealed the method is very effective.

For future work, there are several important questions that remain unresolved.  First, making a CDA method work for the more 
commonly used second order mixed formulation is an important next step.  Second, the analytical results herein give sufficient conditions on $H$ and $\omega$ for the results to hold, but our numerical tests suggest these conditions are not sharp.  Hence an improved analysis that sharpens these bounds may be possible.  Finally, extending these results to two-phase flow is an important future direction.

% the BIB
%\bibliographystyle{plain}
%\bibliography{CHDA,graddiv,references,LariosBiblio}

\begin{thebibliography}{10}

\bibitem{Albanez_Nussenzveig_Lopes_Titi_2016}
D.~Albanez, H.~Nussenzveig~Lopes, and E.~Titi.
\newblock Continuous data assimilation for the three-dimensional
  {N}avier--{S}tokes-$\alpha$ model.
\newblock {\em Asymptotic Anal.}, 97(1-2):139--164, 2016.

\bibitem{Altaf_Titi_Knio_Zhao_Mc_Cabe_Hoteit_2015}
M.~U. {Altaf}, E.~S. {Titi}, O.~M. {Knio}, L.~{Zhao}, M.~F. {McCabe}, and
  I.~{Hoteit}.
\newblock Downscaling the 2{D} {B}enard convection equations using continuous
  data assimilation.
\newblock {\em Comput. Geosci}, 21(3):393--410, 2017.

\bibitem{Anthes_1974_JAS}
R.~A. Anthes.
\newblock Data assimilation and initialization of hurricane prediction models.
\newblock {\em J. Atmos. Sci.}, 31(3):702--719, 1974.

\bibitem{ADSV:16:CHvirtual}
P.~F. Antonietti, L.~B. Da~Veiga, S.~Scacchi, and M.~Verani.
\newblock A {{$C^1$}} virtual element method for the {Cahn--Hilliard} equation
  with polygonal meshes.
\newblock {\em SIAM J. Numer. Anal.}, 54(1):34--56, 2016.

\bibitem{AKW:15:adaptiveDGCH}
A.~C. Aristotelous, O.~A. Karakashian, and S.~M. Wise.
\newblock Adaptive, second-order in time, primitive-variable discontinuous
  {Galerkin} schemes for a {Cahn--Hilliard} equation with a mass source.
\newblock {\em IMA J. Numer. Anal.}, 35(3):1167--1198, 2015.

\bibitem{Azouani2014ContinuousDA}
A.~Azouani, E.~Olson, and E.~S. Titi.
\newblock Continuous data assimilation using general interpolant observables.
\newblock {\em Journal of Nonlinear Science}, 24:277--304, 2014.

\bibitem{Bessaih_Olson_Titi_2015}
H.~Bessaih, E.~Olson, and E.~S. Titi.
\newblock Continuous data assimilation with stochastically noisy data.
\newblock {\em Nonlinearity}, 28(3):729--753, 2015.

\bibitem{Biswas_Martinez_2017}
A.~Biswas and V.~R. Martinez.
\newblock Higher-order synchronization for a data assimilation algorithm for
  the 2{D} {N}avier--{S}tokes equations.
\newblock {\em Nonlinear Anal. Real World Appl.}, 35:132--157, 2017.

\bibitem{brenner:11:frontiers}
S.~C. Brenner.
\newblock {$C^0$} interior penalty methods.
\newblock In {\em Frontiers in Numerical Analysis-Durham 2010}, pages 79--147.
  Springer, 2011.

\bibitem{Cahn:61:spinodal}
J.~W. Cahn.
\newblock On spinodal decomposition.
\newblock {\em Acta Metall Mater}, 9(9):795--801, 1961.

\bibitem{CH:58:free}
J.~W. Cahn and J.~E. Hilliard.
\newblock Free energy of a nonuniform system. {I.} interfacial free energy.
\newblock {\em J. Chem. Phys.}, 28(2):258--267, 1958.

\bibitem{CS:2018:CHPhaseField}
Y.~Cai and J.~Shen.
\newblock Error estimates for a fully discretized scheme to a cahn-hilliard
  phase-field model for two-phase incompressible flows.
\newblock {\em Mathematics of Computation}, 87(313):2057--2090, 2018.

\bibitem{Chen:19:CHdirect}
L.~Chen.
\newblock Direct solver for the {Cahn--Hilliard} equation by
  {Legendre--Galerkin} spectral method.
\newblock {\em J. Comput. Appl. Math.}, 358:34--45, 2019.

\bibitem{CLSWW:2018:CHWillmore}
Y.~Chen, J.~Lowengrub, J.~Shen, C.~Wang, and S.M. Wise.
\newblock Efficient energy stable schemes for isotropic and strongly
  anisotropic cahn--hilliard systems with the willmore regularization.
\newblock {\em Journal of Computational Physics}, 365:56--73, 2018.

\bibitem{CMW:2011:MinCahnHilliard}
R.~Choksi, M.~Maras, and J.F. Williams.
\newblock {2D} phase diagram for minimizers of a {C}ahn--{H}illiard functional
  with long-range interactions.
\newblock {\em SIAM Journal on Applied Dynamical Systems}, 10(4):1344--1362,
  2011.

\bibitem{Daley_1993_atmospheric_book}
R.~Daley.
\newblock {\em Atmospheric {D}ata {A}nalysis}.
\newblock Cambridge Atmospheric and Space Science Series. Cambridge University
  Press, 1993.

\bibitem{DHKTLD19}
S.~Desamsetti, I.~Hoteit, O.~Knio, E.~Titi, S.~Langodan, and H.~Prasad Dasari.
\newblock Efficient dynamical downscaling of general circulation models using
  continuous data assimilation.
\newblock {\em Quarterly Journal of the Royal Meteorological Society},
  145(724):3175--3194, 2019.

\bibitem{DFW:15:CHDS}
A.~E. Diegel, X.~H. Feng, and S.~M. Wise.
\newblock Analysis of a mixed finite element method for a
  {Cahn--Hilliard--Darcy--Stokes} system.
\newblock {\em SIAM J. Numer. Anal.}, 53(1):127--152, 2015.

\bibitem{Farhat_Jolly_Titi_2015}
A.~Farhat, M.~S. Jolly, and E.~S. Titi.
\newblock Continuous data assimilation for the 2{D} {B}\'enard convection
  through velocity measurements alone.
\newblock {\em Phys. D}, 303:59--66, 2015.

\bibitem{FLT19}
A.~Farhat, E.~Lunasin, and E.~Titi.
\newblock A data assimilation algorithm: The paradigm of the 3{D}
  {L}eray-$\alpha$ model of turbulence.
\newblock {\em Partial Differential Equations Arising from Physics and
  Geometry}, pages 253--273, 2019.

\bibitem{Farhat_Lunasin_Titi_2016abridged}
A.~Farhat, E.~Lunasin, and E.~S. Titi.
\newblock Abridged continuous data assimilation for the 2{D} {N}avier--{S}tokes
  equations utilizing measurements of only one component of the velocity field.
\newblock {\em J. Math. Fluid Mech.}, 18(1):1--23, 2016.

\bibitem{Farhat_Lunasin_Titi_2016benard}
A.~Farhat, E.~Lunasin, and E.~S. Titi.
\newblock Data assimilation algorithm for 3{D} {B\'e}nard convection in porous
  media employing only temperature measurements.
\newblock {\em J. Math. Anal. Appl.}, 438(1):492--506, 2016.

\bibitem{Feng:2006:NSCH}
X.~Feng.
\newblock Fully discrete finite element approximations of the
  {Navier--Stokes--Cahn-Hilliard} diffuse interface model for two-phase fluid
  flows.
\newblock {\em SIAM Journal on Numerical Analysis}, 44(3):1049--1072, 2006.

\bibitem{Foias_Mondaini_Titi_2016}
C.~Foias, C.~F. Mondaini, and E.~S. Titi.
\newblock A discrete data assimilation scheme for the solutions of the
  two-dimensional {N}avier-{S}tokes equations and their statistics.
\newblock {\em SIAM J. Appl. Dyn. Syst.}, 15(4):2109--2142, 2016.

\bibitem{GNT18}
B.~Garcia-Archilla, J.~Novo, and E.~Titi.
\newblock Uniform in time error estimates for a finite element method applied
  to a downscaling data assimilation algorithm.
\newblock {\em SIAM Journal on Numerical Analysis}, 58:410--429, 2020.

\bibitem{Hoke_Anthes_1976_MWR}
J.~Hoke and R.~Anthes.
\newblock The initialization of numerical models by a dynamic-initialization
  technique.
\newblock {\em Monthly Weather Review}, 104(12):1551--1556, 1976.

\bibitem{Ibdah_Mondaini_Titi_2018uniform}
H.~Ibdah, C.~Mondaini, and E.~Titi.
\newblock Fully discrete numerical schemes of a data assimilation algorithm:
  uniform-in-time error estimates.
\newblock {\em IMA Journal of Numerical Analysis}, 11 2019.
\newblock drz043.

\bibitem{Jolly_Martinez_Titi_2017}
M.~Jolly, V.~Martinez, and E.~Titi.
\newblock A data assimilation algorithm for the subcritical surface
  quasi-geostrophic equation.
\newblock {\em Adv. Nonlinear Stud.}, 17(1):167--192, 2017.

\bibitem{Kalman_1960_JBE}
R.~Kalman.
\newblock A new approach to linear filtering and prediction problems.
\newblock {\em J. Basic Eng.}, 82(1):35--45, 1960.

\bibitem{LRZ19}
A.~Larios, L.~Rebholz, and C.~Zerfas.
\newblock Global in time stability and accuracy of {IMEX-FEM} data assimilation
  schemes for {Navier-Stokes} equations.
\newblock {\em Computer Methods in Applied Mechanics and Engineering},
  345:1077--1093, 2019.

\bibitem{Law_Stuart_Zygalakis_2015_book}
K.~Law, A.~Stuart, and K.~Zygalakis.
\newblock {\em A {M}athematical {I}ntroduction to {D}ata {A}ssimilation},
  volume~62 of {\em Texts in Applied Mathematics}.
\newblock Springer, Cham, 2015.

\bibitem{LLG:2002:HeleShaw}
H-G. Lee, J.S. Lowengrub, and J.~Goodman.
\newblock Modeling pinchoff and reconnection in a hele-shaw cell. i. the models
  and their calibration.
\newblock {\em Physics of Fluids}, 14(2):492--513, 2002.

\bibitem{DMB18}
P.~C.~Di Leoni, A.~Mazzino, and L.~Biferale.
\newblock Inferring flow parameters and turbulent configuration with
  physics-informed data assimilation and spectral nudging.
\newblock {\em Physical Review Fluids}, 3(104604), 2018.

\bibitem{DMB20}
P.~C.~Di Leoni, A.~Mazzino, and L.~Biferale.
\newblock Synchronization to big data: nudging the {N}avier-{S}tokes equations
  for data assimilation of turbulent flows.
\newblock {\em Physical Review X}, 10(011023), 2020.

\bibitem{Lunasin_Titi_2015}
E.~Lunasin and E.~S. Titi.
\newblock Finite determining parameters feedback control for distributed
  nonlinear dissipative systems---a computational study.
\newblock {\em Evol. Equ. Control Theory}, 6(4):535--557, 2017.

\bibitem{Markowich_Titi_Trabelsi_2016}
P.~Markowich, E.~S. Titi, and S.~Trabelsi.
\newblock Continuous data assimilation for the three-dimensional
  {B}rinkman--{F}orchheimer-extended {D}arcy model.
\newblock {\em Nonlinearity}, 29(4):1292, 2016.

\bibitem{Miranville:19:CH}
A.~Miranville.
\newblock {\em The {Cahn--Hilliard} Equation: Recent Advances and
  Applications}.
\newblock SIAM, 2019.

\bibitem{R:16:firedrake}
F.~Rathgeber, D.A. Ham, L.~Mitchell, M.~Lange, F.~Luporini, A.T.T. McRae, G-T.
  Bercea, G.R. Markall, and P.H.J. Kelly.
\newblock Firedrake: automating the finite element method by composing
  abstractions.
\newblock {\em ACM Transactions on Mathematical Software (TOMS)}, 43(3):1--27,
  2016.

\bibitem{RZ19}
L.~Rebholz and C.~Zerfas.
\newblock Simple and efficient continuous data assimilation of evolution
  equations via algebraic nudging.
\newblock {\em Numerical Methods for Partial Differential Equations},
  37(3):2588--2612, 2021.

\bibitem{SZ90}
R.~Scott and S.~Zhang.
\newblock Finite element interpolation of nonsmooth functions satisfying
  boundary conditions.
\newblock {\em Mathematics of Computation}, 54(190):483--493, 1990.

\bibitem{Thomee:84:galerkin}
V.~Thom{\'e}e.
\newblock {\em Galerkin finite element methods for parabolic problems}, volume
  1054.
\newblock Springer, 1984.

\bibitem{GT:15:survey}
G.~Tierra and F.~Guill{\'e}n-Gonz{\'a}lez.
\newblock Numerical methods for solving the {Cahn-Hilliard} equation and its
  applicability to related energy-based models.
\newblock {\em Arch. Comput. Method E}, 22(2):269--289, 2015.

\bibitem{vTBVL:2009:PFCM}
S.~van Teeffelen, R.~Backofen, A.~Voigt, and H.~L{\"o}wen.
\newblock Derivation of the phase-field-crystal model for colloidal
  solidification.
\newblock {\em Physical Review E}, 79(5):051404, 2009.

\bibitem{ZRSI19}
C.~Zerfas, L.~Rebholz, M.~Schneier, and T.~Iliescu.
\newblock Continuous data assimilation reduced order models of fluid flow.
\newblock {\em Computer Methods in Applied Mechanics and Engineering},
  357(112596):1--21, 2019.

\end{thebibliography}

\end{document}